\documentclass[twoside]{article} 
\usepackage{fullpage}
\usepackage{amsmath,amsbsy,amsfonts,amssymb,amsthm}
\usepackage{graphicx}
\usepackage{natbib}
\usepackage{algorithm, algorithmic}
\usepackage{color}
\usepackage{array}

\theoremstyle{plain}
\newtheorem{theorem}{Theorem}[section]
\newtheorem{claim}[theorem]{Claim}
\newtheorem{corollary}[theorem]{Corollary}

\newtheorem{lemma}[theorem]{Lemma}

\theoremstyle{definition}

\newcommand{\be}{\begin{equation}}
\newcommand{\ee}{\end{equation}}

\newcommand{\vect}[1]{\ensuremath{\mathbf{#1}}}
\newcommand{\mat}[1]{\ensuremath{\mathbf{#1}}}
\newcommand{\grad}{\bigtriangledown}

\newcommand{\twonorm}[1]{\left\| {#1} \right\|_2}

\newcommand{\frob}[1]{\left\|#1\right\|_F}
\newcommand{\order}[1]{\ensuremath{\mathcal{O}\left(#1\right)}}

\newcommand{\Om}[1]{\ensuremath{\Omega\left(#1\right)}}

\newcommand{\eqdef}{\stackrel{\triangle}{=}}
\newcommand{\abs}[1]{\left|#1\right|}

\newcommand{\inv}[1]{{#1}^{-1}}
\newcommand{\trans}[1]{{#1}^{\top}}

\def\nn{\nonumber}

\newcommand\R{\mathbb{R}}

\newcommand\ch{\widehat{c}}

\newcommand\Rnn{\mathbb{R}^{n\times n}}

\newcommand{\singmin}[1]{\ensuremath{\sigma_{\min}\left(#1\right)}}

\newcommand{\singi}[1]{\ensuremath{\sigma_{i}\left(#1\right)}}
\newcommand{\lammin}[1]{\ensuremath{\lambda_{\min}\left(#1\right)}}
\newcommand{\lami}[1]{\ensuremath{\lambda_{i}\left(#1\right)}}

\newcommand{\poly}[1]{\textrm{poly}\left(#1\right)}

\newcommand{\M}{\mat{M}}

\newcommand{\A}{\mat{A}}

\newcommand{\Errt}[1][t]{\mat{\bigtriangleup_{#1}}}
\newcommand\cnM{\kappa}
\newcommand\compfactor{\alpha}
\newcommand\compfactorbeta{\beta}
\newcommand{\cn}[1]{\kappa\left(#1\right)}
\newcommand\X{\mat{X}}
\newcommand\fstar{f_*}

\newcommand\ut[1][t]{{u_{#1}}}

\newcommand\U{\mat{U}}

\newcommand{\Ut}[1][t]{\mat{U_{#1}}}
\newcommand{\Utinv}{\inv{\Ut}}

\newcommand\Utp{\mat{U_{t+1}}}

\newcommand\Utptild{\mat{\widetilde{U}_{t+1}}}

\newcommand{\eye}{\mat{I}}

\renewcommand\v{\vect{v}}

\newcommand\w{\vect{w}}
\newcommand\wTr{\trans{\vect{w}}}

\newcommand\matmult{{\omega}}

\begin{document}

% If your paper is accepted and the title of your paper is very long,
% the style will print as headings an error message. Use the following
% command to supply a shorter title of your paper so that it can be
% used as headings.
%
%\runningtitle{I use this title instead because the last one was very long}

% If your paper is accepted and the number of authors is large, the
% style will print as headings an error message. Use the following
% command to supply a shorter version of the authors names so that
% they can be used as headings (for example, use only the surnames)
%
%\runningauthor{Surname 1, Surname 2, Surname 3, ...., Surname n}

\title{Global Convergence of Non-Convex Gradient Descent for Computing Matrix Squareroot}

\author{
Prateek Jain\footnote{Microsoft Research, India. Email: prajain@microsoft.com} \and
Chi Jin\footnote{University of California, Berkeley. Email: chijin@cs.berkeley.edu} \and 
Sham M. Kakade\footnote{University of Washington. Email: sham@cs.washington.edu} \and
Praneeth Netrapalli\footnote{Microsoft Research, India. Email: praneeth@microsoft.com}
}

\maketitle
% \aistatstitle{}

% \aistatsauthor{  \And Chi Jin \And Sham M. Kakade \And Praneeth Netrapalli }

% \aistatsaddress{ \And UC Berkeley \And University of Washington \And  Microsoft Research India} 

%!TEX root = noncvx_gradDes.tex

\begin{abstract}
While there has been a significant amount of work studying gradient descent techniques for non-convex optimization problems over the last few years, all existing results establish either \emph{local convergence with good rates} or \emph{global convergence with highly suboptimal rates}, for many problems of interest. In this paper, we take the first step in getting the best of both worlds -- establishing global convergence and obtaining a good rate of convergence for the problem of computing squareroot of a positive definite (PD) matrix, which is a widely studied problem in numerical linear algebra with applications in machine learning and statistics among others.

%In this paper, we study the problem of computing square-root of a positive semidefinite (PSD) matrix $\M$, i.e., compute $\U^*$ s.t. $\U^*=\arg\min_{\U}\|\M-\U^2\|_F^2$. We provide the {\em first global convergence guarantees} for a natural algorithm based on iteratively updating $\U$ via gradient descent, from {\em arbitrary starting point} $\Ut[0]$.
Given a PD matrix $\M$ and a PD starting point $\Ut[0]$, we show that gradient descent with appropriately chosen step-size finds an $\epsilon$-accurate squareroot of $\M$ in $\order{\compfactor
\log (\frob{\M-\Ut[0]^2}/\epsilon)}$ iterations, where $\compfactor \eqdef (\max\{\twonorm{\Ut[0]}^2,\twonorm{\M}\}$ $/$ $\min \{\sigma_{\min}^2(\Ut[0]),\singmin{\M}\} )^{3/2}$. Our result is the first to establish global convergence for this problem and that it is robust to errors in each iteration. 
A key contribution of our work is the general proof technique which we believe should further excite research in understanding deterministic and stochastic variants of simple non-convex gradient descent algorithms with \emph{good global convergence rates} for other problems in machine learning and numerical linear algebra.
%Existing provable non-convex results typically only guarantee local convergence and hence require careful initialization. In contrast, our result shows global convergence from an arbitrary initial point, iterating through the space with a large number of saddle points. %More generally, our result demonstrates that non-convex optimization can be a viable approach to obtaining fast, robust algorithms, taking a step towards understanding and validating the efficacy of such methods in various large scale practical applications. In contrast to previous work on related non-convex problems, which provide only local convergence guarantees and hence require careful initialization, our result is one of the first such result which shows global convergence from an arbitrary initial point, thus requiring navigating through the space with a large number of local minimas. Our proof technique is fairly general and we believe it should further excite research in designing deterministic and stochastic variants of simple non-convex gradient descent algorithms for other problems in machine learning and numerical linear algebra.
\end{abstract}
%!TEX root = main.tex

\section{Introduction}\label{sec:intro}
Given that a large number of problems and frameworks in machine learning are non-convex optimization problems (examples include non-negative matrix factorization~\citep{LeeS2001}, sparse coding~\citep{AharonEB2006}, matrix sensing~\citep{recht2010guaranteed}, matrix completion~\citep{KorenBV2009}, phase retrieval~\citep{netrapalli2015phase} etc.), in the last few years, there has been an increased interest in designing efficient non-convex optimization algorithms. Several recent works establish \emph{local convergence} to the global optimum for problems such as matrix sensing~\citep{jain2013low,tu2015low}, matrix completion~\citep{jain2014fast,sun2015guaranteed}, phase retrieval~\citep{candes2015phase}, sparse coding~\citep{agarwal2013learning} and so on (and hence, require careful initialization). However, despite strong empirical evidence, none of these results have been able to establish \emph{global convergence}. On the other hand some other recent works~\citep{nesterov2006cubic,ge2015escaping,lee2016gradient,sun2015nonconvex} establish the global convergence of gradient descent methods to local minima for a large class of non-convex problems but the results they obtain are quite suboptimal compared to the local convergence results mentioned above. In other words, results that have very good rates are only local (and results that are global do not have very good rates).

Therefore, a natural and important question is if gradient descent actually has a \emph{good global convergence rate} when applied to specific and important functions that are of interest in machine learning. Apart from theoretical implications, such a result is also important in practice since a) finding a good initialization might be difficult and b) local convergence results are inherently difficult to extend to stochastic algorithms due to noise.

In this work, we answer the above question in affirmative for the problem of computing square root of a positive definite (PD) matrix $\M$: i.e., $\min_{\U\succeq 0} f(\U)$ where $f(\U)=\|\M-\U^2\|_F^2$. This problem in itself is a fundamental one and arises in several contexts such as computation of the matrix sign function \citep{Higham2008},
%(see \cite{Higham2008}, Chapter 6), 
computation of data whitening matrices, signal processing applications \citep{KaminskiBS1971,Carlson1990,VanDerMerweW2001,TippettABHW2003} and so on. %Additionally, several important numerical linear algebra problems like computation of the matrix sign function, the definite generalized eigenvalue problem etc., also reduce to computing matrix square-root .
%Computation of the matrix squareroot is a fundamental problem in several applications.
\subsection{Related work}
Given the importance of computing the matrix squareroot, there has been a tremendous amount of work in the numerical linear algebra community focused on this problem \citep{BjorckH1983,Higham1986,Higham1987,Higham1997,Meini2004}. For a detailed list of references, see Chapter 6 in Higham's book \citep{Higham2008}. 

The basic component of most these algorithms is the Newton's method to find the square root of a positive number. Given a positive number $m$ and a positive starting point $\ut[0]$, Newton's method gives rise to the iteration
\begin{align}
\ut[t+1] \leftarrow \frac{1}{2}\left(\ut+\frac{m}{\ut}\right).
\label{eqn:newton-scalar}
\end{align}
It can be shown that the iterates converge to $\sqrt{m}$ at a quadratic rate (i.e., $\epsilon$-accuracy in $\log \log \frac{1}{\epsilon}$ iterations). The extension of this approach to the matrix case is not straight forward due to non commutativity of matrix multiplication. For instance, if $\M$ and $\Ut$ were matrices, it is not clear if $\frac{m}{\ut}$ should be replaced by $\Utinv\M$ or $\M\Utinv$ or something else. One approach to overcome this issue is to select $\Ut[0]$ carefully to ensure commutativity through all iterations \citep{Higham1986,Higham1997,Meini2004}, for example, $\Ut[0]=\M$ or $\Ut[0]=\eye$. % which ensures that all the iterates have the same eigenvectors as $\M$ and hence commute with $\M$. 
However, commutativity is a brittle property and small numerical errors in an iteration itself can result in loss of commutativity. Although a lot of work since, has focused on designing stable iterations that are inspired by Eq.\eqref{eqn:newton-scalar} \citep{Higham1986,Higham1997,Meini2004}, and has succeeded in making it robust in practice, no provable robustness guarantees are known in the presence of repeated errors.
Similarly, another recent approach by ~\citet{Sra15} uses geometric optimization to solve the matrix squareroot problem but their analysis also does not address the stability or robustness to numerical or statistical errors (if we see a noisy version of $\M$) .

Another approach to solve the matrix square-root problem is to use the eigenvalue decomposition (EVD) and then take square-root of the eigenvalues. To the best of our knowledge, state-of-the-art computation complexity for computing the EVD of a matrix (in the real arithmetic model of computation) is due to~\citet{PanCZ1998}, which is
%$\order{n^3 + n\log^2 n \log \log \frac{1}{\epsilon}}$ for general matrices and
$\order{n^{\matmult}\log n + n \log^2 n \log \log \frac{1}{\epsilon}}$ for matrices with distinct eigenvalues. Though the result is close to optimal (in reducing the EVD to matrix multiplication), the algorithm and the analysis are quite complicated. For instance robustness of these methods to errors is not well understood. As mentioned above however, our focus is to understand if local search techniques like gradient descent (which are often applied to several non-convex optimization procedures) indeed avoid saddle points and local minima, and can guide the solution to global optimum. 
% Though the algorithm seems to be robust to noise, making this precise in the presence of round off errors for instance 
%Note that there is a drastic difference in the above two bounds depending on whether the matrix has repeated eigenvalues or not (with the bound being cubic in $n$ in the general case). It is not clear if this dichotomy is inherent in the computation of the squareroot as well (our result shows that this is indeed not the case). More importantly, the result is for the real arithmetic model of computation, and it is not clear if the method is robust to round off errors, as are incurred in practice. To the best of our knowledge, the best known bound for computing the eigen decomposition in finite precision models of computation is $\order{n^3\log \cnM}$, which is by power method \cite{GolubVL2012}. Here $\cnM$ denotes the condition number of the input matrix $\M$ (i.e., the ratio of the largest to the smallest eigenvalue of $\M$).

As we mentioned earlier, \citet{ge2015escaping, lee2016gradient} give some recent results on global convergence for general non-convex problems which can be applied to matrix squareroot problem. While \citet{lee2016gradient} prove only asymptotic behavior of gradient descent without any rate, applying the result of~\citet{ge2015escaping} gives us a runtime of $\order{n^{10}/\poly{\epsilon}}$\footnote{For optimization problem of dimension $d$, \citet{ge2015escaping} proves convergence in the number of iteration of $\order{d^4}$, with $\order{d}$ computation per iteration. In matrix squareroot problem $d=n^2$, which gives total $\order{n^{10}}$ dependence.}, which is highly suboptimal in terms of its dependence on $n$ and $\epsilon$.

Finally, we note that subsequent to this work, \citet{jin2017escape} proved global convergence results with almost sharp dimension dependence for a much wider class of functions. While \citet{jin2017escape} explicitly add perturbation to help escape saddle points, our framework does not require perturbation, and shows that for this problem, gradient descent naturally stays away from saddle points.

\subsection{Our contribution}\label{sec:contrib}
\renewcommand{\arraystretch}{1.5}
\begin{table*}[t]
	\centering
	\begin{tabular}{ | >{\centering\arraybackslash} m{5.5cm} | >{\centering\arraybackslash} m{4.9cm} | >{\centering\arraybackslash} m{2cm} | >{\centering\arraybackslash} m{1.8cm} |}
		\hline
		\textbf{Method} & \textbf{Runtime} & Global convergence & Provable robustness\\
		\hline
		\textbf{Gradient descent (this paper)} & $\order{\compfactor n^{\matmult}\log \frac{1}{\epsilon}}$ & $\checkmark$ & $\checkmark$ \\
		\hline
		{Stochastic gradient descent \citep{ge2015escaping}} & $\order{n^{10}/\poly{\epsilon}}$ & $\checkmark$ & $\checkmark$ \\
		\hline
		Newton variants \citep{Higham2008}& $\order{n^{\matmult}\log \log \frac{1}{\epsilon}}$ & $\times$ & $\times$ \\
		\hline
		EVD (algebraic \citep{PanCZ1998})& $\order{n^{\matmult}\log n + n \log^2 n \log \log \frac{1}{\epsilon}}$ & Not iterative & $\times$ \\
		\hline
		EVD (power method \citep{GolubVL2012})& $\order{n^3 \log \frac{1}{\epsilon}}$ & Not iterative & $\times$ \\
		\hline
	\end{tabular}
	\caption{Comparison of our result to existing ones. Here $\omega$ is the matrix multiplication exponent and $\alpha$ is our convergence rate parameter defined in Eq.\eqref{eqn:compfactor}. We show that our method enjoys global convergence and is also provably robust to arbitrary bounded errors in each iteration. In contrast, Newton variants only have local convergence and their robustness to errors in multiple iterations is not known. Robustness of methods based on eigenvalue decomposition is also not well understood.}
	\label{tab:comparison}
\end{table*}
In this paper, we propose doing gradient descent on the following non-convex formulation:
\begin{align}\label{eqn:main}
\min_{\U\in \Rnn; \U \succeq 0} \frob{\M-\U^2}^2.
\end{align}
We show that if the starting point $\Ut[0]$ is chosen to be a positive definite matrix, our algorithm converges to the global optimum of Eq.\eqref{eqn:main} at a geometric rate. In order to state our runtime, we make the following notation:
\begin{align}\label{eqn:compfactor}
	\compfactor \eqdef \left(\frac{\max\left(\twonorm{\Ut[0]},\sqrt{\twonorm{\M}}\right)}{\min \left(\singmin{\Ut[0]},\sqrt{\singmin{\M}}\right)}\right)^3,
\end{align}
where $\singmin{\Ut[0]}$ and $\twonorm{\Ut[0]}$ are the minimum singular value and operator norm respectively of the starting point $\Ut[0]$, and $\singmin{\M}$ and $\twonorm{\M}$ are those of $\M$. Our result says that gradient descent converges $\epsilon$ close to the optimum of Eq.\eqref{eqn:main} in $\order{\alpha \log \frac{\frob{\M - \Ut[0]^2}}{\epsilon}}$ iterations.
Each iteration involves doing only three matrix multiplications and no inversions or leastsquares. So the total runtime of our algorithm is $\order{n^{\matmult} \alpha \log \frac{\frob{\M}}{\epsilon}}$, where $\omega < 2.373$ is the matrix multiplication exponent\citep{williams2012multiplying}. As a byproduct of our global convergence guarantee, we obtain the robustness of our algorithm to numerical errors \emph{for free}. In particular, we show that our algorithm is robust to errors in multiple steps in the sense that if each step has an error of at most $\delta$, then our algorithm achieves a limiting accuracy of $\order{\alpha\sqrt{\twonorm{\M}} \delta}$. Another nice feature of our algorithm is that it is based purely on matrix multiplications, where as most existing methods require matrix inversion or solving a system of linear equations.
An unsatisfactory part of our result however is the dependence on $\alpha \geq \cnM^{3/2}$, where $\cnM$ is the condition number of $\M$. We prove a lower bound of $\Om{\cnM}$ iterations for our method which tells us that the dependence on problem parameters in our result is not a weakness in our analysis.

\textbf{Outline}: In Section~\ref{sec:notation}, we will briefly set up the notation we will use in this paper. In Section~\ref{sec:results}, we will present our algorithm, approach and main results. We will present the proof of our main result in Section~\ref{sec:proofsketch} and conclude in Section~\ref{sec:conc}. The proofs of remaining results can be found in the Appendix.

%%% Local Variables: 
%%% mode: latex
%%% TeX-master: "noncvx_gradDes"
%%% End: 

\section{Notation}\label{sec:notation}
Let us briefly introduce the notation we will use in this paper. We use boldface lower case letters ($\v,\w,\ldots$)
to denote vectors and boldface upper case letters ($\M,\X,\ldots$) to denote matrices. $\M$ denotes the input matrix
we wish to compute the squareroot of. $\singi{\A}$ denotes the $i^{\textrm{th}}$ singular value of $\A$. $\singmin{\A}$ denotes the smallest singular value of $\A$. $\cn{\A}$ denotes
the condition number of $\A$ i.e., $\frac{\twonorm{\A}}{\singmin{\A}}$. $\cnM$ without an argument denotes $\cn{\M}$. $\lambda_i\left(\A\right)$ denotes the $i^{\textrm{th}}$ largest eigenvalue of $\A$ and $\lambda_{\textrm{min}}(\A)$ denotes the smallest eigenvalue of $\A$.

%\section{Preliminaries}
%In this draft, we consider the problem of computing the square root of a given $n \times n$ PSD matrix $\M$ as $\U^2$. We propose to do this by doing gradient descent on the following non convex optimization problem:
%%\begin{align}\label{eqn:optprob}
%%\min_{\U\in \Rnn} \frob{\M-\U^2}^2.
%%\end{align}
%In the next section, we will obtain convergence rate for gradient descent on the above optimization problem. The following lists some of the notation we use in this paper.
%\begin{itemize}
%\item	$\M$ -- input PSD matrix
%\item	$\U^2$ -- candidate factorization
%\item	$\cnM$ -- condition number of $\M$ i.e., $\frac{\twonorm{\M}}{\singmin{\M}}$
%\end{itemize}

%!TEX root = noncvx_gradDes.tex

\section{Our Results}\label{sec:results}
In this section, we present our guarantees and the high-level approach for the analysis of Algorithm~\ref{algo:gradDesMS} which is just 
gradient descent on the non-convex optimization problem:
\begin{align}\label{eqn:optprob}
\min_{\U\in \Rnn;\U \succeq 0} \frob{\M-\U^2}^2.
\end{align}
%where we constrain $U$ to be symmetric.
%will present our algorithm, approach and guarantees.
%\subsection{Algorithm}
%Our algorithm is just  The algorithm is formally presented in Algorithm~\ref{algo:gradDesMS}.
\begin{algorithm}[t]
\caption{Gradient descent for matrix square root}\label{algo:gradDesMS}
\begin{algorithmic}
\renewcommand{\algorithmicrequire}{\textbf{Input: }}
\renewcommand{\algorithmicensure}{\textbf{Output: }}
\REQUIRE $\M$, PD matrix $\Ut[0], \eta, T$
\ENSURE  $\U$
\FOR{$t = 0,\cdots,T-1$}
\STATE $\Utp = \Ut - \eta \left(\Ut^2 - \M\right) \Ut - \eta \Ut\left(\Ut^2 - \M\right)$
\ENDFOR
\STATE \textbf{Return} $\Ut[T]$.
\end{algorithmic}
\end{algorithm}
We first present a warmup analysis, where we assume that all the iterates of Algorithm~\ref{algo:gradDesMS} commute with $\M$. Later, in Section~\ref{sec:approach} we present our approach to analyze Algorithm~\ref{algo:gradDesMS} for any general starting point $\Ut[0]$. We provide formal guarantees in Section~\ref{sec:guarantees}. 
\subsection{Warmup -- Analysis with commutativity}
In this section, we will give a short proof of convergence for Algorithm~\ref{algo:gradDesMS}, when we ensure that all iterates commute with $\M$.
\begin{lemma}\label{lem:conv-comm}
There exists a constant $c$ such that if $\eta < \frac{c}{\twonorm{\M}}$, and $\Ut[0]$ is chosen to be $\sqrt{\twonorm{\M}}\cdot\eye$, then $\Ut$ in Algorithm~\ref{algo:gradDesMS} satisfies:
\begin{align*}
\frob{\Ut^2 - \M}^2 &\leq \exp\left(-2\eta \singmin{\M} t\right) \frob{\Ut[0]^2 - \M}^2.
\end{align*}
\end{lemma}
\begin{proof}
Since $\Ut[0]=\sqrt{\twonorm{\M}}\eye$ has the same eigenvectors as $\M$, it can be seen by induction that $\Ut$ has the same eigenvectors as $\M$ for every $t$. Every singular value $\singi{\Utp}$ can be written as
\begin{align}\label{eqn:sing-comm}
\singi{\Utp} = \left(1-2\eta\left(\singi{\Ut}^2-\singi{\M}\right)\right)\singi{\Ut}.
\end{align}

Firstly, this tells us that $\twonorm{\Ut} < \sqrt{2\twonorm{\M}}$ for every $t$. Verifying this is easy using induction. The statement holds for $t=0$ by hypothesis. Assuming it holds for $\Ut$, the induction step follows by considering the two cases $\twonorm{\Ut} \leq \sqrt{\twonorm{\M}}$ and $\sqrt{\twonorm{\M}} < \twonorm{\Ut} < \sqrt{2\twonorm{\M}}$ separately and using the assumption that $\eta < \frac{c}{\twonorm{\M}}$. A similar induction argument also tells us that $\singi{\Ut} > \sqrt{\frac{\singi{\M}}{2}}$. 
Eq.\eqref{eqn:sing-comm} can now be used to yield the following convergence equation:
\begin{align*}
& \abs{\singi{\Utp}^2 - \singi{\M}} \\
= &\abs{\singi{\Ut}^2-\singi{\M}}\cdot\left(1-4\eta\singi{\Ut}^2 \right. \\
& \quad\quad\quad \left. + 4 \eta^2 \singi{\Ut}^2 \left(\singi{\Ut}^2-\singi{\M}\right)\right)\\
\leq & \abs{\singi{\Ut}^2-\singi{\M}}\cdot \left(1-4\eta\singi{\Ut}^2  \right. \\
& \quad\quad\quad \left. + 8\eta^2 \singi{\Ut}^2 \twonorm{\M}\right) \\
\leq & \left(1-2\eta \singi{\Ut}^2\right) \abs{\singi{\Ut}^2-\singi{\M}} \\
\leq & \exp\left(- \eta \singmin{\M} \right) \abs{\singi{\Ut}^2-\singi{\M}},
\end{align*}
where we used the hypothesis on $\eta$ in the last two steps. Using induction gives us
\begin{align*}
\abs{\singi{\Ut}^2-\M} \leq \exp\left(-\eta \singmin{\M} t\right) \abs{\singi{\Ut[0]}^2-\M}.
\end{align*}
This can now be used to prove the lemma:
\begin{align*}
& \frob{\Ut^2 - \M}^2 = \sum_i \left(\singi{\Ut}^2 - \singi{\M}\right)^2 \\
\leq & \exp\left(-2\eta \singmin{\M} t\right) \sum_i \left(\singi{\Ut[0]}^2 - \singi{\M}\right)^2 \\
\leq & \exp\left(-2\eta \singmin{\M} t\right) \frob{\Ut[0]^2 - \M}^2.
\end{align*}
\end{proof}
Note that the above proof crucially used the fact that the eigenvectors of $\Ut$ and $\M$ are {\em aligned}, to reduce the matrix iterations to iterations only over the singular values. 
\subsection{Approach}\label{sec:approach}
As we begin to investigate the global convergence properties of Eq.\eqref{eqn:optprob}, the above argument breaks down
due to lack of alignment between the singular vectors of $\M$ and those of the iterates $\Ut$.
Let us now take a step back and consider non-convex optimization in general. There are two broad reasons why local
search approaches fail for these problems. The first is the presence of local minima and the second is the presence
of saddle points. Each of these presents different challenges: with local minima, local search approaches have no way
of certifying whether the convergence point is a local minimum or global minimum; while with saddle points, if the
iterates get close to a saddle point, the local neighborhood looks essentially flat and escaping the saddle point
may take exponential time.

The starting point of our work is the realization that the non-convex formulation of the matrix squareroot problem does not have any local minima. This can be argued using the continuity of the matrix squareroot function, and this statement is indeed true for many matrix factorization problems. The only issue to be contended with is the presence of saddle points. In order to overcome this issue, it suffices to show that the iterates of the algorithm never get too close to a saddle point. More concretely, while optimizing a function $f$ with iterates $\Ut$, it suffices to show that for every $t$, $\Ut$ always stay in some region $\mathcal{D}$ far from saddle points so that for all $\U, \U' \in \mathcal{D}$:
\begin{align}
&\frob{\grad f(\U) - \grad f(\U')} \leq L\frob{\U - \U'}  \label{eqn:hessF-cond} \\
&\frob{\grad f(\U)} \geq \sqrt{\ell \left( f(\U)-\fstar \right)}, \label{eqn:gradF-cond}
\end{align}
where $\fstar = \min_\U f(\U)$, and $L$ and $\ell$ are some constants. 
If we flatten matrix $\U$ to be $n^2$-dimensional vector, then Eq.\eqref{eqn:hessF-cond} is the standard smoothness assumption in optimization, and Eq.\eqref{eqn:gradF-cond} is known as gradient dominated property \citep{polyak1963gradient,nesterov2006cubic}.
If Eq.\eqref{eqn:hessF-cond} and Eq.\eqref{eqn:gradF-cond} hold, it follows from standard analysis that gradient descent with a step size $\eta < \frac{1}{L}$ achieves geometric convergence with
\begin{align*}
f(\Ut)-\fstar \leq \exp\left(-\eta \ell t/2\right) \left(f(\Ut[0])-\fstar\right).
\end{align*}
Since the gradient in our setting is $\left(\Ut^2 - \M\right) \Ut +\Ut\left(\Ut^2 - \M\right)$, in order to establish Eq.\eqref{eqn:gradF-cond}, it suffices to lower bound $\lammin{\Ut[t]}$. Similarly, in order to establish Eq.\eqref{eqn:hessF-cond}, it suffices to upper bound $\twonorm{\Ut[t]}$. Of course, we cannot hope to converge if we start from a saddle point. In particular Eq.\eqref{eqn:gradF-cond} will not hold for any $l > 0$.
The core of our argument consists of Lemmas~\ref{lem:opnorm-bound-MS} and~\ref{lem:smalleigval}, which essentially establish Eq.\eqref{eqn:hessF-cond} and Eq.\eqref{eqn:gradF-cond} respectively for the matrix squareroot problem Eq.\eqref{eqn:optprob}, with the resulting parameters $l$ and $L$ dependent on the starting point $\Ut[0]$. Lemmas~\ref{lem:opnorm-bound-MS} and~\ref{lem:smalleigval} accomplish this by proving upper and lower bounds respectively on $\twonorm{\Ut[t]}$ and $\lammin{\Ut[t]}$.
The proofs of these lemmas use only elementary linear algebra and we believe such results should be possible for many more matrix factorization problems.

%%% Local Variables: 
%%% mode: latex
%%% TeX-master: "noncvx_gradDes"
%%% End: 

%!TEX root = noncvx_gradDes.tex

\subsection{Guarantees}\label{sec:guarantees}
In this section, we will present our main results establishing that gradient descent on~\eqref{eqn:optprob} converges to the matrix square root at a geometric rate and its robustness to errors in each iteration.
\subsubsection{Noiseless setting}\label{sec:matsqrroot-noiseless}
The following theorem establishes geometric convergence of Algorithm~\ref{algo:gradDesMS} from a full rank initial point.
\begin{theorem}\label{thm:geomconv}
There exist universal numerical constants $c$ and $\ch$ such that if $\Ut[0]$ is a PD matrix and $\eta < \frac{c}{\compfactor \compfactorbeta^2 }$, then for every $t \in [T-1]$, we have $\Ut$ be a PD matrix with
\begin{align*}
\frob{\M - \Ut^2} \leq \exp\left(- \ch \eta \compfactorbeta^2 t\right) \frob{\M - \Ut[0]^2},
\end{align*}
where $\compfactor$ and $\compfactorbeta$ are defined as
\begin{align*}
	\compfactor &\eqdef \left(\frac{\max\left(\twonorm{\Ut[0]},\sqrt{\twonorm{\M}}\right)}{\min \left(\singmin{\Ut[0]},\sqrt{\singmin{\M}}\right)}\right)^3, \\
	\compfactorbeta &\eqdef \min \left(\singmin{\Ut[0]},\sqrt{\singmin{\M}}\right)
\end{align*}
\end{theorem}
\textbf{Remarks}:
\begin{itemize}
\item	This result implies global geometric convergence. Choosing $\eta = \frac{c}{\compfactor \compfactorbeta^2}$, in order to obtain an accuracy of $\epsilon$, the number of iterations required would be $\order{\compfactor \log \frac{\frob{\M - \Ut[0]^2}}{\epsilon}}$.
\item	 Note that saddle points of~\eqref{eqn:optprob} must be rank degenerate matrix ($\sigma_{\min}(\U) = 0$) and starting Algorithm~\ref{algo:gradDesMS} from a point close to the rank degenerate surface takes a long time to get away from the saddle surface. Hence, as $\Ut[0]$ gets close to being rank degenerate, convergence rate guaranteed by Theorem~\ref{thm:geomconv} degrades (as $\cn{\Ut[0]}^3$). It is possible to obtain a smoother degradation with a finer analysis, but in the current paper, we trade off optimal results for a simple analysis.
\item	The convergence rate guaranteed by Theorem~\ref{thm:geomconv} also depends on the relative scales of $\Ut[0]$ and $\M$ (say as measured by $\twonorm{\Ut[0]}^2/\twonorm{\M}$) and is best if it is close to $1$.
\item  We believe that it is possible to extend our analysis to the case where $\M$ is low rank (PSD). In this case, suppose $\text{rank}(\M) = k$, and let $\U^\star$ be the k-dimensional subspace in which $\M$ resides. Then, saddle points should satisfy $\sigma_k(\trans\U \U^\star) = 0$.
%\item	Finding a matrix $\Ut[0]$ satisfying the hypothesis of Theorem~\ref{thm:geomconv} is easy. For instance, we first compute $\lambda$ such that  and choose $\Ut[0] = \lambda \eye$.
%\item	The constants $\sqrt{3}$ and $\frac{1}{10}$ in the bounds on $\twonorm{\Ut[0]}$ and $\singmin{\Ut[0]}$ are arbitrary. Choosing any other constants in their place would only change the requirement on learning rate by a constant and give the same result.
\end{itemize}
A simple corollary of this result is when we choose $\Ut[0] = \lambda \eye$, where $\twonorm{\M} \leq \lambda \leq 2 \twonorm{\M}$ (such a $\lambda$ can be found in time $\order{n^2}$~\cite{musco2015stronger}). 
%the runtime of Algorithm~\ref{algo:gradDesMS} is $\order{n^{\matmult} \cnM^{3/2} \log \frac{\frob{\M - \Ut[0]^2}}{\epsilon}}$, where $\matmult$ is the matrix multiplication exponent and $\cnM$ is the condition number of $\M$.
\begin{corollary}
	Suppose we choose $\Ut[0] = \lambda \eye$, where $\twonorm{\M} \leq \lambda \leq 2 \twonorm{\M}$. Then $\frob{\M - \Ut[T]^2} \leq \epsilon$ for $T \geq \order{\cnM^{\frac{3}{2}}\log \frac{\frob{\M - \Ut[0]^2}}{\epsilon}}$.
\end{corollary}

\subsubsection{Noise Stability}\label{sec:matsqrroot-stability}
Theorem~\ref{thm:geomconv} assumes that the gradient descent updates are performed with out any error. This is not practical. For instance, any implementation of Algorithm~\ref{algo:gradDesMS} would incur rounding errors. Our next result addresses this issue by showing that Algorithm~\ref{algo:gradDesMS} is stable in the presence of small, arbitrary errors in each iteration. This will establish the stability of our algorithm in the presence of round-off errors for instance. Formally, we consider in every gradient step, we incur an error $\Errt$. 
% Algorithm~\ref{algo:gradDesMS-err}.

% \begin{algorithm}[t]
% \caption{Gradient descent with errors for matrix square root}\label{algo:gradDesMS-err}
% \begin{algorithmic}
% \renewcommand{\algorithmicrequire}{\textbf{Input: }}
% \renewcommand{\algorithmicensure}{\textbf{Output: }}
% \REQUIRE $\M$, PSD matrix $\Ut[0], \eta, T$
% \ENSURE  $\U$
% \FOR{$t = 0,\cdots,T-1$}
% \STATE $\Utp = \Ut - \eta \left(\Ut^2 - \M\right) \Ut - \eta \Ut\left(\Ut^2 - \M\right) + \Errt$
% \ENDFOR
% \STATE \textbf{Return} $\Ut[T]$.
% \end{algorithmic}
% \end{algorithm}

The following theorem shows that as long as the errors $\Errt$ are small enough, Algorithm~\ref{algo:gradDesMS} recovers the true squareroot upto an accuracy of the error floor. The proof of the theorem follows fairly easily from that of Theorem~\ref{thm:geomconv}.
\begin{theorem}\label{thm:stability}
There exist universal numerical constants $c$ and $\ch$ such that the following holds:
Suppose $\Ut[0]$ is a PD matrix and $\eta < \frac{c}{\compfactor\compfactorbeta^2}$ where $\compfactor$ and $\compfactorbeta$ are defined as before:
\begin{align*}
	\compfactor &\eqdef \left(\frac{\max\left(\twonorm{\Ut[0]},\sqrt{\twonorm{\M}}\right)}{\min \left(\singmin{\Ut[0]},\sqrt{\singmin{\M}}\right)}\right)^3, \\
	\compfactorbeta &\eqdef \min \left(\singmin{\Ut[0]},\sqrt{\singmin{\M}}\right).
\end{align*}
 Suppose further that $\twonorm{\Errt}< \frac{1}{300}\eta {\singmin{\M}} \compfactorbeta$. Then, for every $t \in [T-1]$, we have $\Ut$ be a PD matrix with
\begin{align*}
&\frob{\M - \Ut^2} 
\leq \exp\left(- \ch \eta \compfactorbeta^2  t\right) \frob{\M - \Ut[0]^2} \\
&\quad + 4 \max(\twonorm{\Ut[0]},\sqrt{3 \twonorm{\M}}) \sum_{s=0}^{t-1} e^{- \ch \eta \compfactorbeta^2 (t-s-1)} \frob{\Errt[s]}.
\end{align*}
\end{theorem}
\textbf{Remarks}:
\begin{itemize}
%\item	As remarked after Theorem~\ref{thm:geomconv}, finding a $\Ut[0]$ satisfying the hypothesis of Theorem~\ref{thm:stability} is straightforward by computing $\twonorm{\M}$ approximately and starting with, for instance, a scaled version of $\eye$.
\item	Since the errors above are multiplied by a decreasing sequence, they can be bounded to obtain a limiting accuracy of $\order{\compfactor (\twonorm{\Ut[0]}+\sqrt{\twonorm{\M}}) (\sup_s \frob{\Errt[s]})}$.
%\begin{align*}
%&4 \max(\twonorm{\Ut[0]},\sqrt{3 \twonorm{\M}}) \sum_{s=0}^{t-1} \exp\left(- \ch \eta \min(\singmin{\Ut[0]}^2, {\singmin{\M}}) (t-s-1) \right) \frob{\Errt[s]} \\
%&\leq 4 \max(\twonorm{\Ut[0]},\sqrt{3 \twonorm{\M}}) \left(\sum_{s=0}^{\infty} \exp\left(- \ch \eta \min(\singmin{\Ut[0]}^2, {\singmin{\M}}) (t-s-1) \right)\right) \sup_s \frob{\Errt[s]} \\
%&\leq 4 \ch \left(\frac{\min(\singmin{\Ut[0]}^2, {\singmin{\M}}) }{\max(\twonorm{\Ut[0]}^2, {\twonorm{\M}})}\right)^{3/2}  \max(\twonorm{\Ut[0]},\sqrt{3 \twonorm{\M}}) (\sup_s \frob{\Errt[s]}).
%\end{align*}
%This means that Algorithm~\ref{algo:gradDesMS-err} achieves a limiting accuracy of
%\begin{align*}
%	\order{\left(\frac{\min(\singmin{\Ut[0]}^2, {\singmin{\M}}) }{\max(\twonorm{\Ut[0]}^2, {\twonorm{\M}})}\right)^{3/2}  \max(\twonorm{\Ut[0]},\sqrt{3 \twonorm{\M}}) (\sup_s \frob{\Errt[s]})}.
%\end{align*}
% We are not aware of any other matrix squareroot algorithms that run in matrix multiplication time and are robust to errors in multiple iterations in this fashion.
\item	If there is error in only the first iteration i.e., $\Errt = 0$ for $t \neq 0$, then the initial error $\Errt[0]$ is attenuated with every iteration,
\begin{align*}
&\frob{\M-\Ut^2} \leq \exp\left(- \ch \eta \compfactorbeta^2 t\right) \frob{\M - \Ut[0]^2} \\
&\quad\quad\quad + 6 \max(\twonorm{\Ut[0]}^2, {\twonorm{\M}}) e^{- \ch \eta \compfactorbeta^2 (t-1)} \frob{\Errt[0]}.
\end{align*}
That is, our dependence on $\frob{\Errt[0]}$ is exponentially decaying with respect to time $t$. On the contrary, best known results only guarantees the error dependence on $\frob{\Errt[0]}$ will not increase significantly with respect to time $t$ \citep{Higham2008}.

% Contrast this to the stability results for Newton variants where the best known results state that error in a single iteration is not amplified by more than a constant factor \cite{Higham2008}.

%\item	Just as in Theorem~\ref{thm:geomconv}, if we choose $\eta = \frac{c}{\sqrt{\cnM} \twonorm{\M}}$, in order to obtain an accuracy of $\epsilon$, the number of iterations required would be $\order{\cnM^{3/2} \log \frac{\frob{\M - \Ut[0]^2}}{\epsilon}}$. So the total time
%complexity of the algorithm is $\order{n^{\matmult} \cnM^{3/2} \log \frac{\frob{\M - \Ut[0]^2}}{\epsilon}}$, where $\matmult$ is the matrix multiplication exponent.
%\item	The constants $\sqrt{3}$ and $\frac{1}{10}$ in the bounds on $\twonorm{\Ut[0]}$ and $\singmin{\Ut[0]}$ are arbitrary. Choosing any other constants in their place would only change the requirement on the learning rate by a constant.
\end{itemize}

\subsubsection{Lower Bound}\label{sec:lowerbound}
We also prove the following lower bound showing that gradient descent with a fixed step size requires $\Om{\cnM}$ iterations to achieve an error of $\order{\singmin{\M}}$.
% \begin{lemma}\label{lem:lb}
% There exists PSD matrices $\M$, and $\Ut[0]$ with $\twonorm{\Ut[0]} \leq \sqrt{3 \twonorm{\M}}$ and $\singmin{\Ut[0]} \geq \frac{1}{10}\sqrt{\singmin{\M}}$, and constants $c$ and $\ct$ such that
% \begin{itemize}
% \item	if we choose step size $\eta \geq \frac{c}{\twonorm{\M}}$, then for all $t\ge 1$, $\frob{\Ut^2 - \M} \ge \twonorm{\M}$, and
% \item	if we choose step size $\eta \leq \frac{c}{\twonorm{\M}}$, we will have $\frob{\Ut^2 - \M} \ge \frac{1}{4}\singmin{\M}$ for all 
% $t \le \ct \cdot \cnM$.
% \end{itemize}
% \end{lemma}
\begin{theorem}\label{thm:lb}
For any value of $\cnM$, we can find a matrix $\M$ such that, for any step size $\eta$, there exists an initialization $\Ut[0]$ that has the same eigenvectors as $\M$, with $\twonorm{\Ut[0]} \leq \sqrt{3 \twonorm{\M}}$ and $\singmin{\Ut[0]} \geq \frac{1}{10}\sqrt{\singmin{\M}}$, such that
we will have $\frob{\Ut^2 - \M} \ge \frac{1}{4}\singmin{\M}$ for all 
$t \le \cnM$.
\end{theorem}
This lemma shows that the convergence rate of gradient descent fundamentally depends on the condition number $\cnM$, even if we start with a matrix that has the same eigenvectors and similar scale as $\M$. In this case, note that the lower bound of Theorem~\ref{thm:lb} is off from the upper bound of Theorem~\ref{thm:geomconv} by $\sqrt{\cnM}$.
Though we do not elaborate in this paper, it is possible to formally show that a dependence of $\cnM^{3/2}$ is the best
bound possible using our argument (i.e., one along the lines of Section~\ref{sec:approach}).
\section{Proof Sketch for Theorem~\ref{thm:geomconv}} \label{sec:proofsketch}
In this section, we will present the proof of Theorem~\ref{thm:geomconv}. To make our strategy more concrete and transparent, we will leave the full proofs of some technical lemmas in Appendix \ref{sec:matfact-result}.

% Our framework can be applied to problems with global minima and saddle
At a high level, our framework consists of following three steps:
\begin{enumerate}
\item Show all bad stationary points lie in a measure zero set $\{\U| \phi(\U) = 0\}$ for some constructed potential function $\phi(\cdot)$. In this paper, for the matrix squareroot problem, we choose the potential function $\phi(\cdot)$ to be the smallest singular value function $\singmin{\cdot}$.

\item Prove for any $\epsilon>0$, if initial $\U_0 \in \mathcal{D}_\epsilon = \{\U | ~|\phi(\U)| > \epsilon\}$ and the stepsize is chosen appropriately, then we have all iterates $\U_t \in \mathcal{D}_\epsilon$. That is, updates will always keep away from bad stationary points.

\item Inside regions $\mathcal{D}_\epsilon$, show that the optimization function satisfies good properties such as smoothness and gradient-dominance, which establishes convergence to a global minimum with good rate.
\end{enumerate}
Since we can make $\epsilon$ arbitrarily small and since $\{\U| \phi(\U) = 0\}$ is a measure zero set, this essentially establishes convergence from a (Lebesgue) measure one set, proving global convergence. 

We note that step 2 above implies that no stationary point found in the set $\{\U| \phi(\U) = 0\}$ is a local minimum -- it must either be a saddle point or a local maximum. This is because starting at any point outside $\{\U| \phi(\U) = 0\}$ does not converge to $\{\U| \phi(\U) = 0\}$. Therefore, our framework can be mostly used for non-convex problems with saddle points but no spurious local minima.

% In step 3, we show both smoothness and gradient-dominance \citep{polyak1963gradient,nesterov2006cubic} which lead to geometric convergence.

Before we proceed with the full proof, we will first illustrate the three steps above for a simple, special case where $n=2$ and all relevant matrices are diagonal. Specifically, we choose target matrix $\M$ and parameterize $\U$ as:
\begin{align*}
\M = \begin{pmatrix}
4& 0 \\ 0&2
\end{pmatrix},
\quad\quad
\U = \begin{pmatrix}
x& 0 \\ 0&y
\end{pmatrix}.
\end{align*}
Here $x$ and $y$ are unknown parameters. Since we are concerned with $\U \succeq 0$, we see that $x,y \geq 0$. The reason we restrict ourselves to diagonal matrices is so that the parameter space is two dimensional letting us give a good visual representation of the parameter space.
Figures~\ref{fig:contour} and~\ref{fig:flow} show the plots of function value contours and negative gradient flow respectively as a function of $x$ and $y$.
%For this easier version of problem, we show the contour of objective function $f(\U) = \frob{\U^2-\M}^2$ in Figure \ref{fig:contour}. The orange region indicates larger objective value compared to the blue region. The gradient flow of objective function is also shown in Figure \ref{fig:flow}, where the arrows are pointed toward negative gradient direction.

\begin{figure}[t]
\centering
\begin{minipage}{.45\textwidth}
  \centering
\includegraphics[width=\columnwidth]{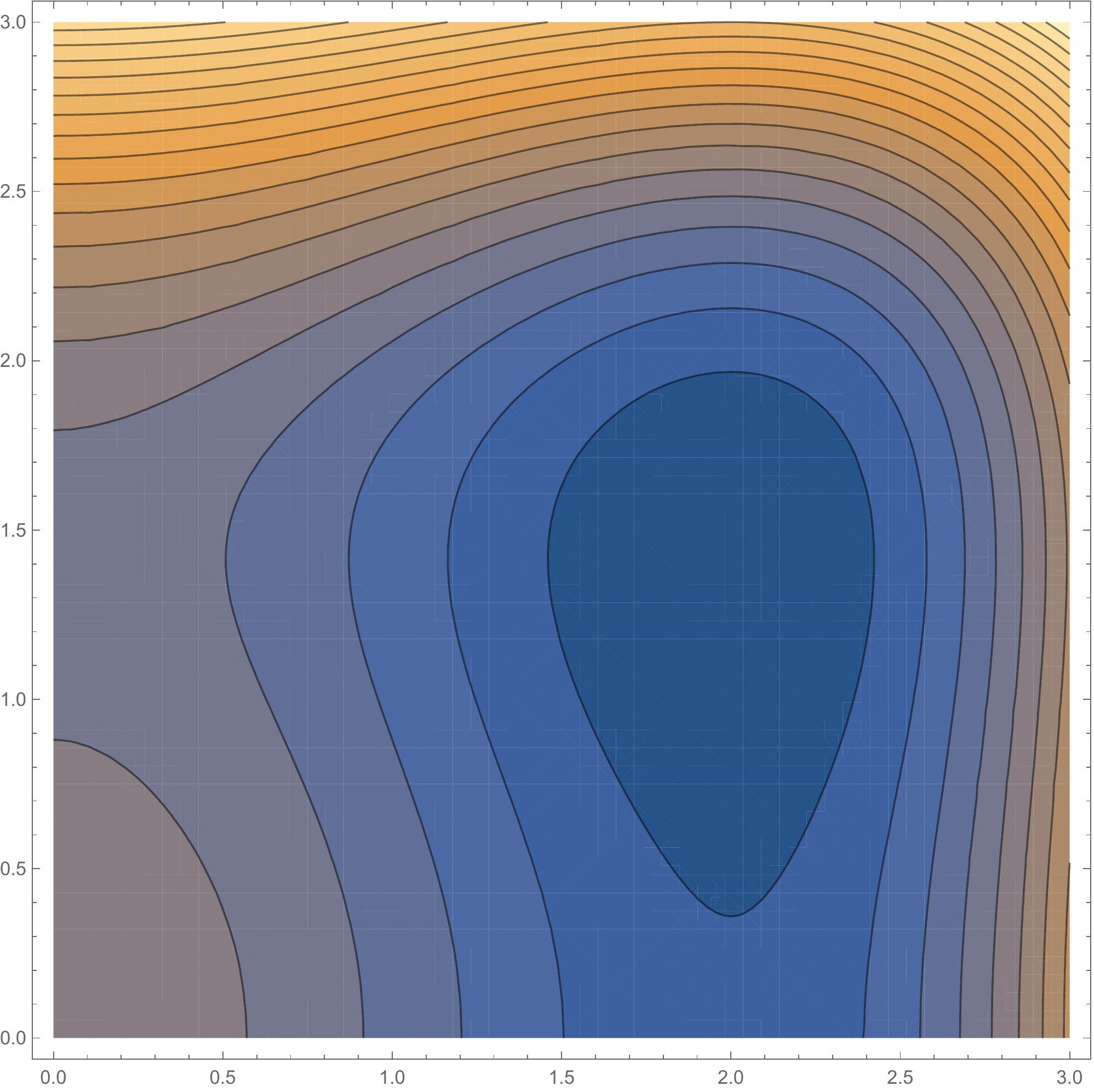}
\caption{Contour of Objective Functions}
\label{fig:contour}
\end{minipage}%
\begin{minipage}{.05\textwidth}
~
\end{minipage}
\begin{minipage}{.45\textwidth}
  \centering
\includegraphics[width=\columnwidth]{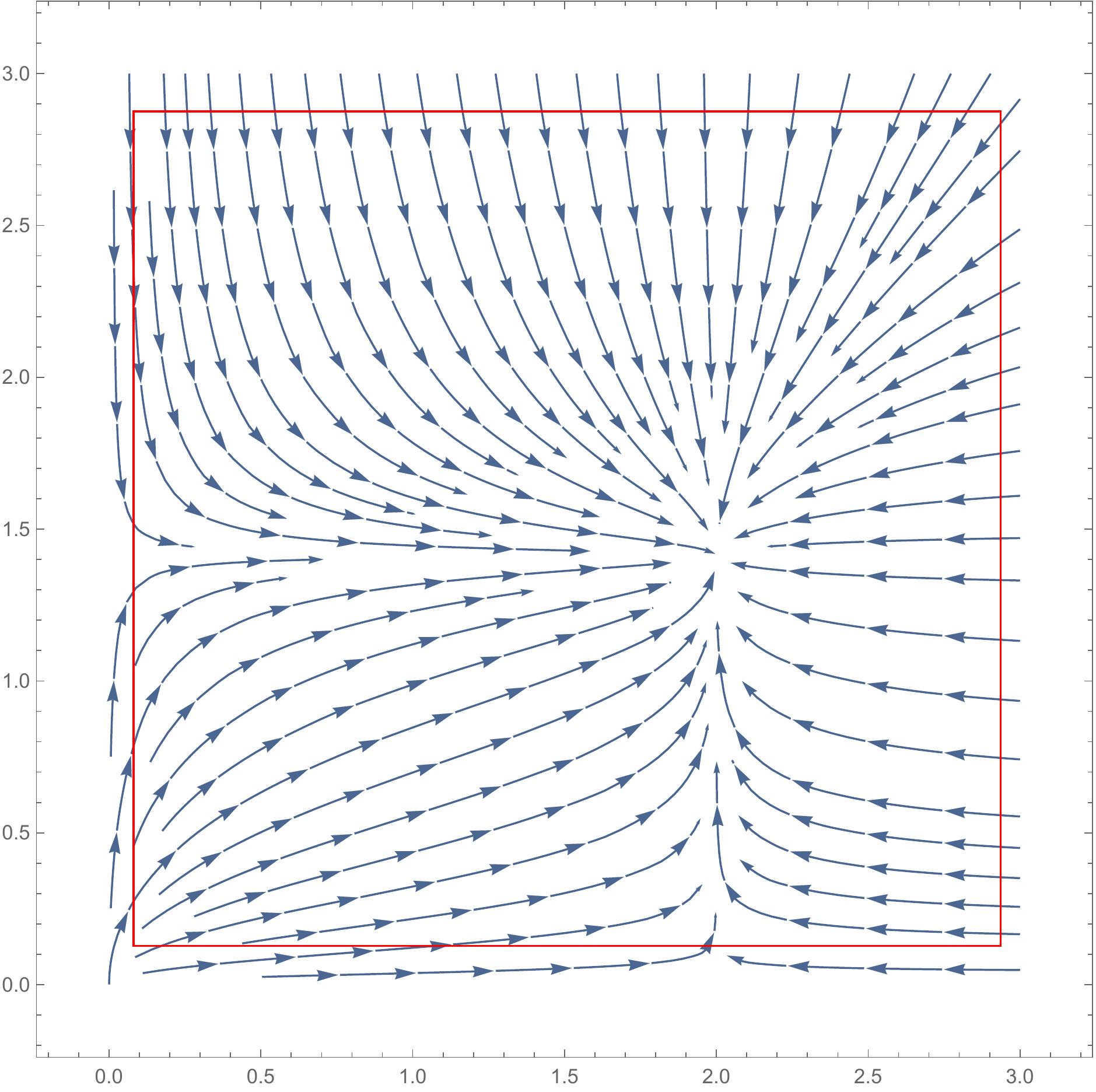}
\caption{Flow of Negative Gradient}
\label{fig:flow}
\end{minipage}
\end{figure}

% \begin{figure}[t]
% \vskip 0.2in
% \begin{center}
% \centerline{\includegraphics[width=\columnwidth]{fig/contour.pdf}}
% \caption{Contour of Objective Functions}
% \label{fig:contour}
% \end{center}
% \vskip -0.2in
% \end{figure} 

% \begin{figure}[t]
% \vskip 0.2in
% \begin{center}
% \centerline{\includegraphics[width=\columnwidth]{fig/flow2.pdf}}
% \caption{Flow of Negative Gradient}
% \label{fig:flow}
% \end{center}
% \vskip -0.2in
% \end{figure} 

We will use Figures~\ref{fig:contour} and~\ref{fig:flow} to qualitatively establish the three steps in our framework.

\begin{enumerate}
\item From Figure~\ref{fig:contour}, we note that $(2, \sqrt{2})$ is the global minimum. $(2, 0), (0, \sqrt{2})$ are saddle points, while $(0, 0)$ is local maximum. We notice all the stationary points which are not global minima lie on the surface $\sigma_{\min}(\U) =0$, that is, the union of x-axis and y-axis. 

\item By defining a boundary $\{\U | \sigma_{\min}(\U)>c, \twonorm{\U}<C\}$ for some small $c$ and large $C$ (corresponding to the red box in Figure \ref{fig:flow}), we see that negative gradient flow is pointed inside the box which means that for any point in the box, performing gradient descent with a small enough stepsize will ensure that all iterates lie inside the box (and hence keep away from saddle points).

\item Inside the red box, Figure~\ref{fig:flow} shows that negative gradient flow points to the global optimum. Moreover, we can indeed establish upper and lower bounds on the magnitude of gradients within the red box -- this corresponds to establishing smoothness and gradient dominance respectively.
\end{enumerate}

Together, all the above observations along with standard results in optimization tell us that gradient descent has geometric convergence for this problem.

We now present a formal proof of our result.
\subsection{Location of Saddle Points}
We first give a characterization of locations of all the stationary points which are not global minima.

\begin{lemma}\label{lem:saddle-location}
Within symmetric PSD cone $\{\U |\U \succeq 0\}$, all stationary points of $f(\U) = \frob{\M - \U^2}^2$ which are not global minima, must satisfy $\singmin{\U} = 0$.
\end{lemma}

\begin{proof}
For any stationary point $\U'$ of $f(\U)$ which is not on the boundary $\{\U|\sigma_{\min}(\U) = 0\}$, by linear algebra calculation, we have:
\begin{align*}
0 = &\|\nabla f(\U')\|^2_F  = \|(\U'^2 - \M)\U' + \U'(\U'^2 - \M)\|^2_F \\
=& \langle (\U'^2 - \M)\U' + \U'(\U'^2 - \M), \\
&\quad\quad\quad\quad\quad\quad\quad (\U'^2 - \M)\U' + \U'(\U'^2 - \M) \rangle \\
=& 2\text{Tr}([(\U'^2 - \M)\U']^2) + 2\text{Tr}((\U'^2 - \M)^2\U'^2) \\ 
\ge & 4\sigma^2_{\min}(\U') \|\U'^2 - \M\|_F^2
\end{align*}
Therefore, since $\U'$ is not on the boundary of PSD cone, we have $\sigma^2_{\min}(\U) > 0$, which gives
 $f(\U') = \frob{\M - \U'^2}^2 \neq 0$, thus $\U'$ is global minima.
\end{proof}

As mentioned before, note that all the bad stationary points are contained in $\{\U | \singmin{\U} = 0\}$ which is a (Lebesgue) measure zero set.

\subsection{Stay Away from Saddle Surface}
Since the gradient at stationary points is zero, gradient descent can never converge to a global minimum if starting from suboptimal stationary points. 
Fortunately, in our case, gradient updates will keep away from bad stationary points. As in next Lemma, we show that as long as we choose suitable small learning rate, $\singmin{\U_t}$ will never be too small. 

\begin{lemma}\label{lem:smalleigval}
Suppose $\eta < c\frac{\min \left(\singmin{\Ut[0]},\sigma^{1/2}_{\min}(\M)/10\right)}{\max\left(\twonorm{\Ut[0]}^3,\left(3\twonorm{\M}\right)^{3/2}\right)}$, where $c$ is a small enough constant. Then, for every $t \in [T-1]$, we have $\Ut$ in Algorithm \ref{algo:gradDesMS} be a PD matrix with
\begin{align*}
\lammin{\Ut} \geq \min\left(\singmin{\Ut[0]},\frac{\sqrt{\singmin{\M}}}{10}\right).
\end{align*}
\end{lemma}

It turns out that the gradient updates will not only keep $\singmin{\U}$ from being too small, but also keep $\twonorm{\U}$ from being too large.

% Next, we show when initial point $\U_0$ is away from the surface with bad stationary points, all iterates $\U_t$ will also be away from this surface. 

\begin{lemma}\label{lem:opnorm-bound-MS}
%Suppose $\twonorm{\Ut[0]} \leq \sqrt{3 \twonorm{\M}}$ and $\eta < \frac{1}{10\twonorm{\M}}$.
Suppose $\eta < \frac{1}{10 \max\left(\twonorm{\Ut[0]}^2,3{\twonorm{\M}}\right)}$. For every $t\in[T-1]$, we have $\Ut$ in Algorithm \ref{algo:gradDesMS} satisfying:
\begin{align*}
\twonorm{\Ut} \leq \max\left(\twonorm{\Ut[0]}, \sqrt{3 \twonorm{\M}}\right).
\end{align*}
\end{lemma}

Although $\twonorm{\U}$ is not directly related to the surface with bad stationary points, the upper bound on $\twonorm{\U}$ is crucial for the smoothness of function $f(\cdot)$, which gives good convergence rate in Section \ref{sec:sketch-converge}.

\subsection{Convergence in Saddle-Free Region} \label{sec:sketch-converge}
So far, we have been able to establish both upper bounds and lower bounds on singular values of all iterates $\U_t$ given suitable small learning rate. Next, we show that when spectral norm of $\U$ is small, function $f(\U)$ is smooth, and when $\singmin{\U}$ is large, function $f(\U)$ is gradient dominated. 
\begin{lemma} \label{lem:smoothness}
Function $f(\U) = \frob{\M - \U^2}^2$ is $8 \max\{\Gamma, \twonorm{\M}\}$-smooth in region $\{\mat{U} | \twonorm{\mat{U}}^2 \le \Gamma\}$. That is, for any $\Ut[1], \Ut[2] \in \{\mat{U} | \twonorm{\mat{U}}^2 \le \Gamma\}$, we have:
\begin{equation*}
\|\nabla f(\Ut[1]) - \nabla f(\Ut[2])\|_F \le 8 \max\{\Gamma, \twonorm{\M}\} \|\Ut[1] - \Ut[2]\|_F
\end{equation*}
\end{lemma}

\begin{lemma} \label{lem:gradientdominate}
Function $f(\U) = \frob{\M - \U^2}^2$ is $4\gamma$-gradient dominated in region $\{\mat{U} | \singmin{\mat{U}}^2 \ge \gamma\}$. That is, for any $\U \in \{\mat{U} | \singmin{\mat{U}}^2 \ge \gamma\}$, we have:
\begin{equation*}
\|\nabla f(\U)\|^2_F \ge 4\gamma f(\U)
\end{equation*}
\end{lemma}

Lemma \ref{lem:smoothness} and \ref{lem:gradientdominate} are the formal versions of Eq.\eqref{eqn:hessF-cond} and Eq.\eqref{eqn:gradF-cond} in Section \ref{sec:approach}, which are essential in establishing geometric convergence. 

Putting all pieces together, we are now ready prove our main theorem:

\begin{proof}[Proof of Theorem \ref{thm:geomconv}]
Recall the definitions in Theorem \ref{thm:geomconv}:
\begin{align*}
    \compfactor &\eqdef \left(\frac{\max\left(\twonorm{\Ut[0]},\sqrt{\twonorm{\M}}\right)}{\min \left(\singmin{\Ut[0]},\sqrt{\singmin{\M}}\right)}\right)^3, \\
    \compfactorbeta &\eqdef \min \left(\singmin{\Ut[0]},\sqrt{\singmin{\M}}\right)
\end{align*}

By choosing learning rate $\eta < \frac{c}{\compfactor \compfactorbeta^2 }$ with small enough constant $c$. We can satisfy the precondition of Lemma \ref{lem:smalleigval}, and Lemma \ref{lem:opnorm-bound-MS} at same time. Therefore, we know all iterates will fall in region:
\begin{align*}
&\left\{\U \left| \twonorm{\U} \leq \max\left(\twonorm{\Ut[0]}, \sqrt{3 \twonorm{\M}}\right)\right.\right.,\\
&\quad\quad\left.\lammin{\U} \geq \min\left(\singmin{\Ut[0]},\frac{\sqrt{\singmin{\M}}}{10}\right) \right\}
\end{align*}

Then, apply Lemma \ref{lem:smoothness} and Lemma \ref{lem:gradientdominate}, we know in this region, function $f(\U) = \frob{\U^2 - \M}^2$ has smoothness parameter:
\begin{align*}
8 \max\left\{\max\left\{\twonorm{\Ut[0]}^2, 3 \twonorm{\M}\right\}, \twonorm{\M}\right\} \le 24\compfactor^{2/3}\compfactorbeta^2
\end{align*}
and gradient dominance parameter:
\begin{align*}
4\min\left\{\sigma_{\min}^2(\Ut[0]),\frac{\singmin{\M}}{100}\right\} \ge \frac{\beta^2}{25}
\end{align*}
That is, $f(\U)$ in the region is both $24\compfactor^{2/3}\compfactorbeta^2$-smooth, and $\beta^2/25$-gradient dominated. 

Finally, by Taylor's expansion of smooth function, we have:
\begin{align*}
f(\Ut[t+1]) \le& f(\Ut[t]) + \langle \nabla f(\Ut[t]), \Ut[t+1] - \Ut[t] \rangle \\
&\quad\quad\quad\quad\quad+ 12\compfactor^{2/3}\compfactorbeta^2 \frob{\Ut[t+1] - \Ut[t]}^2 \\
=& f(\Ut[t]) - (\eta - 12\eta^2\compfactor^{2/3}\compfactorbeta^2)\frob{\nabla f(\Ut[t])}^2\\
\le & f(\Ut[t]) - \frac{\eta}{2}\frob{\nabla f(\Ut[t])}^2\\
\le &(1-\eta\frac{\compfactorbeta^2}{50})f(\Ut[t])
\end{align*}
The second last inequality is again by setting constant $c$ in learning rate to be small enough, and the last inequality is by the property of gradient dominated. This finishes the proof.

\end{proof}
%!TEX root = noncvx_gradDes.tex

\section{Conclusion}\label{sec:conc}
In this paper, we take a first step towards addressing the large gap between local convergence results with good convergence rates and global convergence results with highly suboptimal convergence rates. We consider the problem of computing the squareroot of a PD matrix, which is a widely studied problem in numerical linear algebra, and show that non-convex gradient descent achieves global geometric convergence with a good rate. In addition, our analysis also establishes the stability of this method to numerical errors. We note that this is the first method to have provable robustness to numerical errors for this problem and our result illustrates that global convergence results are also useful in practice since they might shed light on the stability of optimization methods.%Even for this specific problem, gradient descent is very simple, uses only matrix multiplications, and is perhaps the most natural one, especially in noisy settings. Our result provides the first {\em global convergence} for this algorithm for {\em any} starting point that is not a saddle point. Our analysis naturally extends to provide robustness guarantees in the presence of error in the updates. 
%We believe our arguments provide a framework to extend such results to settings with statistical noise, i.e., when at each step we observe a (highly) noisy version of $M$ (or the gradient). Moreover, our analysis takes us a step closer to analyzing stochastic gradient descent for matrix square-root problem in particular, and matrix decomposition probems in general. We leave this as an interesting topic for future work.  

Our result shows that even in the presence of a large saddle point surface, gradient descent might be able to avoid it and converge to the global optimum at a linear rate. We believe that our framework and proof techniques should be applicable for several other nonconvex problems (especially those based on matrix factorization) in machine learning and numerical linear algebra and would lead to the analysis of gradient descent and stochastic gradient descent in a transparent way while also addressing key issues like robustness to noise or numerical errors.% Finally, this also motivates understanding large scale approaches such as stochastic gradient descent, which might be more appealing in some settings.
%%% Local Variables: 
%%% mode: latex
%%% TeX-master: "noncvx_gradDes"
%%% End: 

\bibliographystyle{plainnat} 
\bibliography{bib}

\begin{thebibliography}{32}
\providecommand{\natexlab}[1]{#1}
\providecommand{\url}[1]{\texttt{#1}}
\expandafter\ifx\csname urlstyle\endcsname\relax
  \providecommand{\doi}[1]{doi: #1}\else
  \providecommand{\doi}{doi: \begingroup \urlstyle{rm}\Url}\fi

\bibitem[Agarwal et~al.(2013)Agarwal, Anandkumar, Jain, and
  Netrapalli]{agarwal2013learning}
Alekh Agarwal, Animashree Anandkumar, Prateek Jain, and Praneeth Netrapalli.
\newblock Learning sparsely used overcomplete dictionaries via alternating
  minimization.
\newblock \emph{arXiv preprint}, 2013.

\bibitem[Aharon et~al.(2006)Aharon, Elad, and Bruckstein]{AharonEB2006}
Michal Aharon, Michael Elad, and Alfred Bruckstein.
\newblock K-{SVD}: An algorithm for designing overcomplete dictionaries for
  sparse representation.
\newblock \emph{Signal Processing, IEEE Transactions on}, 54\penalty0
  (11):\penalty0 4311--4322, 2006.

\bibitem[Bj{\"o}rck and Hammarling(1983)]{BjorckH1983}
{\AA}ke Bj{\"o}rck and Sven Hammarling.
\newblock A {Schur} method for the square root of a matrix.
\newblock \emph{Linear algebra and its applications}, 52:\penalty0 127--140,
  1983.

\bibitem[Candes et~al.(2015)Candes, Li, and Soltanolkotabi]{candes2015phase}
Emmanuel~J Candes, Xiaodong Li, and Mahdi Soltanolkotabi.
\newblock Phase retrieval via wirtinger flow: Theory and algorithms.
\newblock \emph{Information Theory, IEEE Transactions on}, 61\penalty0
  (4):\penalty0 1985--2007, 2015.

\bibitem[Carlson(1990)]{Carlson1990}
Neal Carlson.
\newblock Federated square root filter for decentralized parallel processors.
\newblock \emph{Aerospace and Electronic Systems, IEEE Transactions on},
  26\penalty0 (3):\penalty0 517--525, 1990.

\bibitem[Ge et~al.(2015)Ge, Huang, Jin, and Yuan]{ge2015escaping}
Rong Ge, Furong Huang, Chi Jin, and Yang Yuan.
\newblock Escaping from saddle points—online stochastic gradient for tensor
  decomposition.
\newblock In \emph{Proceedings of The 28th Conference on Learning Theory},
  pages 797--842, 2015.

\bibitem[Golub and Van~Loan(2012)]{GolubVL2012}
Gene~H Golub and Charles~F Van~Loan.
\newblock \emph{Matrix computations}, volume~3.
\newblock JHU Press, 2012.

\bibitem[Higham(1986)]{Higham1986}
Nicholas~J Higham.
\newblock Newton’s method for the matrix square root.
\newblock \emph{Mathematics of Computation}, 46\penalty0 (174):\penalty0
  537--549, 1986.

\bibitem[Higham(1987)]{Higham1987}
Nicholas~J Higham.
\newblock Computing real square roots of a real matrix.
\newblock \emph{Linear Algebra and its applications}, 88:\penalty0 405--430,
  1987.

\bibitem[Higham(1997)]{Higham1997}
Nicholas~J Higham.
\newblock Stable iterations for the matrix square root.
\newblock \emph{Numerical Algorithms}, 15\penalty0 (2):\penalty0 227--242,
  1997.

\bibitem[Higham(2008)]{Higham2008}
Nicholas~J Higham.
\newblock \emph{Functions of matrices: theory and computation}.
\newblock Society for Industrial and Applied Mathematics (SIAM), 2008.

\bibitem[Jain and Netrapalli(2014)]{jain2014fast}
Prateek Jain and Praneeth Netrapalli.
\newblock Fast exact matrix completion with finite samples.
\newblock \emph{arXiv preprint}, 2014.

\bibitem[Jain et~al.(2013)Jain, Netrapalli, and Sanghavi]{jain2013low}
Prateek Jain, Praneeth Netrapalli, and Sujay Sanghavi.
\newblock Low-rank matrix completion using alternating minimization.
\newblock In \emph{Proceedings of the forty-fifth annual ACM symposium on
  Theory of computing}, pages 665--674. ACM, 2013.

\bibitem[Jin et~al.(2017)Jin, Ge, Netrapalli, Kakade, and
  Jordan]{jin2017escape}
Chi Jin, Rong Ge, Praneeth Netrapalli, Sham~M Kakade, and Michael~I Jordan.
\newblock How to escape saddle points efficiently.
\newblock \emph{arXiv preprint arXiv:1703.00887}, 2017.

\bibitem[Kaminski et~al.(1971)Kaminski, Bryson~Jr, and Schmidt]{KaminskiBS1971}
Paul~G Kaminski, Arthur~E Bryson~Jr, and Stanley~F Schmidt.
\newblock Discrete square root filtering: A survey of current techniques.
\newblock \emph{Automatic Control, IEEE Transactions on}, 16\penalty0
  (6):\penalty0 727--736, 1971.

\bibitem[Koren et~al.(2009)Koren, Bell, and Volinsky]{KorenBV2009}
Yehuda Koren, Robert Bell, and Chris Volinsky.
\newblock Matrix factorization techniques for recommender systems.
\newblock \emph{Computer}, \penalty0 (8):\penalty0 30--37, 2009.

\bibitem[Lee and Seung(2001)]{LeeS2001}
Daniel~D Lee and H~Sebastian Seung.
\newblock Algorithms for non-negative matrix factorization.
\newblock In \emph{Advances in neural information processing systems}, pages
  556--562, 2001.

\bibitem[Lee et~al.(2016)Lee, Simchowitz, Jordan, and Recht]{lee2016gradient}
Jason~D Lee, Max Simchowitz, Michael~I Jordan, and Benjamin Recht.
\newblock Gradient descent converges to minimizers.
\newblock \emph{University of California, Berkeley}, 1050:\penalty0 16, 2016.

\bibitem[Meini(2004)]{Meini2004}
Beatrice Meini.
\newblock The matrix square root from a new functional perspective: theoretical
  results and computational issues.
\newblock \emph{SIAM journal on matrix analysis and applications}, 26\penalty0
  (2):\penalty0 362--376, 2004.

\bibitem[Musco and Musco(2015)]{musco2015stronger}
Cameron Musco and Christopher Musco.
\newblock Stronger approximate singular value decomposition via the block
  lanczos and power methods.
\newblock \emph{arXiv preprint}, 2015.

\bibitem[Nesterov and Polyak(2006)]{nesterov2006cubic}
Yurii Nesterov and Boris~T Polyak.
\newblock Cubic regularization of newton method and its global performance.
\newblock \emph{Mathematical Programming}, 108\penalty0 (1):\penalty0 177--205,
  2006.

\bibitem[Netrapalli et~al.(2015)Netrapalli, Jain, and
  Sanghavi]{netrapalli2015phase}
Praneeth Netrapalli, Prateek Jain, and Sujay Sanghavi.
\newblock Phase retrieval using alternating minimization.
\newblock \emph{IEEE Transactions on Signal Processing}, 63\penalty0
  (18):\penalty0 4814--4826, 2015.

\bibitem[Pan et~al.(1998)Pan, Chen, and Zheng]{PanCZ1998}
Victor~Y Pan, Zhao~Q Chen, and Ailong Zheng.
\newblock The complexity of the algebraic eigenproblem.
\newblock \emph{Mathematical Sciences Research Institute, Berkeley}, page~71,
  1998.

\bibitem[Polyak(1963)]{polyak1963gradient}
BT~Polyak.
\newblock Gradient methods for the minimisation of functionals.
\newblock \emph{USSR Computational Mathematics and Mathematical Physics},
  3\penalty0 (4):\penalty0 864--878, 1963.

\bibitem[Recht et~al.(2010)Recht, Fazel, and Parrilo]{recht2010guaranteed}
Benjamin Recht, Maryam Fazel, and Pablo~A Parrilo.
\newblock Guaranteed minimum-rank solutions of linear matrix equations via
  nuclear norm minimization.
\newblock \emph{SIAM review}, 52\penalty0 (3):\penalty0 471--501, 2010.

\bibitem[Sra(2015)]{Sra15}
Suvrit Sra.
\newblock On the matrix square root via geometric optimization.
\newblock \emph{arXiv preprint}, 2015.

\bibitem[Sun et~al.(2015)Sun, Qu, and Wright]{sun2015nonconvex}
Ju~Sun, Qing Qu, and John Wright.
\newblock When are nonconvex problems not scary?
\newblock \emph{arXiv preprint arXiv:1510.06096}, 2015.

\bibitem[Sun and Luo(2015)]{sun2015guaranteed}
Ruoyu Sun and Zhi-Quan Luo.
\newblock Guaranteed matrix completion via nonconvex factorization.
\newblock In \emph{Foundations of Computer Science (FOCS), 2015 IEEE 56th
  Annual Symposium on}, pages 270--289. IEEE, 2015.

\bibitem[Tippett et~al.(2003)Tippett, Anderson, Bishop, Hamill, and
  Whitaker]{TippettABHW2003}
Michael~K Tippett, Jeffrey~L Anderson, Craig~H Bishop, Thomas~M Hamill, and
  Jeffrey~S Whitaker.
\newblock Ensemble square root filters*.
\newblock \emph{Monthly Weather Review}, 131\penalty0 (7):\penalty0 1485--1490,
  2003.

\bibitem[Tu et~al.(2015)Tu, Boczar, Soltanolkotabi, and Recht]{tu2015low}
Stephen Tu, Ross Boczar, Mahdi Soltanolkotabi, and Benjamin Recht.
\newblock Low-rank solutions of linear matrix equations via procrustes flow.
\newblock \emph{arXiv preprint arXiv:1507.03566}, 2015.

\bibitem[Van Der~Merwe and Wan(2001)]{VanDerMerweW2001}
Ronell Van Der~Merwe and Eric Wan.
\newblock The square-root unscented kalman filter for state and
  parameter-estimation.
\newblock In \emph{Acoustics, Speech, and Signal Processing, 2001.
  Proceedings.(ICASSP'01). 2001 IEEE International Conference on}, volume~6,
  pages 3461--3464. IEEE, 2001.

\bibitem[Williams(2012)]{williams2012multiplying}
Virginia~Vassilevska Williams.
\newblock Multiplying matrices faster than coppersmith-winograd.
\newblock In \emph{Proceedings of the forty-fourth annual ACM symposium on
  Theory of computing}, pages 887--898. ACM, 2012.

\end{thebibliography}

\newpage
\appendix
%!TEX root = noncvx_gradDes.tex
\section{Proofs of Lemmas in Section \ref{sec:proofsketch}}\label{sec:matfact-result}
In this section, we restate the Lemmas in Section \ref{sec:proofsketch} which were used to prove Theorem \ref{thm:geomconv}, and present their proofs.

% First, we show the location of all stationary points which are not global minima.
% \begin{lemma}[Restatement of Lemma \ref{lem:saddle-location}]\label{lem:saddle-location-re}
% All stationary points of $f(\U) = \frob{\M - \U^2}^2$ which are not global minima, satisfy $\singmin{\U} = 0$.
% \end{lemma}

% \begin{proof}
% For any stationary point $\U'$, we have:
% \begin{align*}
% 0 = &\|\nabla f(\U')\|^2_F  = \|(\U'^2 - \M)\U' + \U'(\U'^2 - \M)\|^2_F \\
% =& \langle (\U'^2 - \M)\U' + \U'(\U'^2 - \M), (\U'^2 - \M)\U' + \U'(\U'^2 - \M) \rangle \\
% \ge & 4\sigma^2_{\min}(\U') \|\U'^2 - \M\|_F^2
% \end{align*}
% Therefore, if $\U'$ is not global minima, we will have $f(\U') = \frob{\M - \U'^2}^2 \neq 0$, which implies $\sigma_{\min}(\U') = 0$
% \end{proof}

First, we prove lemmas which show a lower bound and a upper bound on the eigenvalues of the intermediate matrices $\Ut$ in Algorithm \ref{algo:gradDesMS}. This shows $\Ut$ always stay away from the surface where unwanted stationary point locate.
\begin{lemma}[Restatement of Lemma \ref{lem:smalleigval}]\label{lem:smalleigval-re}
%Suppose $\Ut[0]$ is a PSD matrix with $\twonorm{\Ut[0]} \leq \sqrt{3 \twonorm{\M}}$, $\singmin{\Ut[0]} \geq \frac{1}{10}\sqrt{\singmin{\M}}$ and $\eta < \frac{c}{\sqrt{\cnM} \twonorm{\M}}$, where $c$ is a small enough constant. Then, for every $t \in [T-1]$, we have $\Ut$ be a PSD matrix with
%\begin{align*}
%\singmin{\Ut} \geq \frac{\sqrt{\singmin{\M}}}{10}.
%\end{align*}
Suppose $\eta < \frac{c\min \left(\singmin{\Ut[0]},\sqrt{\singmin{\M}}/10\right)}{\max\left(\twonorm{\Ut[0]}^3,\left(3\twonorm{\M}\right)^{3/2}\right)}$, where $c$ is a small enough constant. Then, for every $t \in [T-1]$, we have $\Ut$ be a PD matrix with
\begin{align*}
\lammin{\Ut} \geq \min\left(\singmin{\Ut[0]},\frac{\sqrt{\singmin{\M}}}{10}\right).
\end{align*}
\end{lemma}
\begin{proof}
We will prove the lemma by induction. The base case $t=0$ holds trivially. Suppose the lemma holds for some $t$. We will now prove that it holds for $t+1$. We have
%Let $\lammin{\cdot}$ denote the minimum eigenvalue of matrix, then:
\begin{align}
\lammin{\Utp} =& \lammin{\Ut - \eta \left(\Ut^2 - \M\right) \Ut - \eta \Ut\left(\Ut^2 - \M\right)} \nn \\
\ge & \lammin{\frac{3}{4}\Ut - 2 \eta \Ut^3} + \lammin{\frac{1}{4}\Ut + \eta (\M\Ut + \Ut\M)} \nn \\
= & \lammin{\frac{3}{4}\Ut - 2 \eta \Ut^3} + \lammin{\left(\frac{1}{2}\eye + \eta \M\right)\Ut\left(\frac{1}{2}\eye + \eta \M\right) - \eta^2 \M\Ut\M} \label{eqn:lamminUtp-main}
\end{align}
When $\eta \le \frac{1}{100 \max\left(\twonorm{\Ut[0]}^2,3\twonorm{\M}\right)}$, using Lemma~\ref{lem:opnorm-bound-MS} we can bound the first term as
\begin{align}
\lammin{\frac{3}{4}\Ut - 2 \eta \Ut^3} = \frac{3}{4}\singmin{\Ut} - 2\eta \singmin{\Ut}^3.\label{eqn:lamminUtp-1}
\end{align}
To bound the second term, for any vector $\w \in \R^n$ with $\twonorm{\w} =1$, let $\w = \sum_{i}^n \alpha_i \v_i$, where $\v_i$ is the $i^{\textrm{th}}$ eigenvector of $\M$, and $\sum_{i=1}^n \alpha_i^2 = 1$. Then:
\begin{align}
&\wTr\left(\left(\frac{1}{2}\eye + \eta \M \right)\Ut \left(\frac{1}{2}\eye + \eta \M\right) - \eta^2 \M\Ut\M\right)\w \nn \\
\ge & \lammin{\Ut} \twonorm{\left(\frac{1}{2}\eye + \eta \M\right)\w}^2  - \eta^2\twonorm{\Ut} \twonorm{\M\w}^2 \nn \\
= & \lammin{\Ut} \sum_{i=1}^n \left(\frac{1}{2} +\eta \lami{\M} \right)^2\alpha_i^2
- \eta^2 \twonorm{\Ut} \sum_{i=1}^n \lami{\M}^2 \alpha_i^2 \nn \\
= & \lammin{\Ut} \sum_{i=1}^n \alpha_i^2\left(\frac{1}{4} +\eta \lami{\M} + \eta^2\lami{\M} ^2 - \eta^2 \cn{\Ut}\lami{\M}^2 \right) \nn \\
\stackrel{(\zeta_1)}{=} & \singmin{\Ut} \sum_{i=1}^n \alpha_i^2\left(\frac{1}{4} +\eta \singi{\M} + \eta^2\singi{\M} ^2 - \eta^2 \cn{\Ut}\singi{\M}^2 \right) \nn \\
\stackrel{(\zeta_2)}{\ge} & \singmin{\Ut} \sum_{i=1}^n \alpha_i^2\left(\frac{1}{4} +\frac{1}{2}\eta\singi{\M}\right)
\ge \singmin{\Ut}\left(\frac{1}{4} + \frac{1}{2}\eta \singmin{\M}\right), \label{eqn:lamminUtp-2}
\end{align}
where $(\zeta_1)$ is due to the fact that $\Ut$ is a PD matrix, so $\lammin{\Ut} = \singmin{\Ut} \ge 0$, and $(\zeta_2)$ is because since 
%\begin{align*}
$   \eta \le \frac{\min \left(\singmin{\Ut[0]},\sqrt{\singmin{\M}}/10\right)}{\max\left(\twonorm{\Ut[0]}^3,\left(3\twonorm{\M}\right)^{3/2}\right)}
%   \le \frac{1}{100\cn{\Ut} \twonorm{\M}}
    \le \frac{1}{2\cn{\Ut} \twonorm{\M}}$, 
%\end{align*}
we have $\eta^2 \cn{\Ut} \singi{\M}^2 \le \frac{1}{2} \eta \singi{\M}$.

\noindent Plugging Eq.\eqref{eqn:lamminUtp-1} and Eq.\eqref{eqn:lamminUtp-2} into Eq.\eqref{eqn:lamminUtp-main}, we have:
\begin{align*}
\lammin{\Utp} \ge \singmin{\Ut} ( 1 + \frac{1}{2}\eta \singmin{\M} - 2 \eta \singmin{\Ut}^2)
\end{align*}
When $\singmin{\Ut} \le \sqrt{\singmin{\M}}/3$, we obtain: $$\lammin{\Utp} \ge \singmin{\Ut} \ge \max\left(\singmin{\Ut[0]},\frac{\sqrt{\singmin{\M}}}{10}\right),$$ and when $\singmin{\Ut} \ge \sqrt{\singmin{\M}}/3$, we have: \begin{align*}\lammin{\Utp} &\ge \singmin{\Ut} ( 1 - 2 \eta \singmin{\Ut}^2)\\ &\ge \singmin{\Ut} ( 1 - 2 \eta \twonorm{\Ut}^2) \ge \frac{9}{10} \singmin{\Ut} \ge \min\left(\singmin{\Ut[0]},\frac{\sqrt{\singmin{\M}}}{10}\right).\end{align*} This concludes the proof.
\end{proof}

% The following result bounds the operator norm of the intermediate matrices $\Ut$ in Algorithm~\ref{algo:gradDesMS}.

\begin{lemma}[Restatement of Lemma \ref{lem:opnorm-bound-MS}]\label{lem:opnorm-bound-MS-re}
%Suppose $\twonorm{\Ut[0]} \leq \sqrt{3 \twonorm{\M}}$ and $\eta < \frac{1}{10\twonorm{\M}}$.
Suppose $\eta < \frac{1}{10 \max\left(\twonorm{\Ut[0]}^2,3{\twonorm{\M}}\right)}$. For every $t\in[T-1]$, we have:
\begin{align*}
\twonorm{\Ut} \leq \max\left(\twonorm{\Ut[0]}, \sqrt{3 \twonorm{\M}}\right).
\end{align*}
\end{lemma}
\begin{proof}
We will prove the lemma by induction. The base case $t=0$ is trivially true. Supposing the statement is true for $\Ut$, we will prove it for $\Utp$.

Using the update equation of Algorithm~\ref{algo:gradDesMS}, we have:
\begin{align}
\twonorm{\Utp} &= \twonorm{\Ut - \eta \left(\Ut^2 - \M\right) \Ut - \eta \Ut \left(\Ut^2 - \M\right)} \nn\\
 &= \twonorm{\left(\eye - 2 \eta \Ut^2\right) \Ut + \eta \M \Ut + \eta \Ut \M} \nn\\
&\leq \twonorm{\left(\eye - 2 \eta \Ut^2\right) \Ut} + 2 \eta \twonorm{\M} \twonorm{\Ut}. \label{eqn:Utspectralbound}
\end{align}

%Since $\eta < \frac{1}{10\twonorm{\Ut[t]}^2}$, the matrix 
The singular values of the matrix $\left(\eye-2\eta \Ut^2\right)\Ut$ are exactly $(1-2\eta \sigma^2)\cdot{\sigma}$ where $\sigma$ is a singular value of $\Ut$. For $\sigma \leq \sqrt{2 \twonorm{\M}}$, we clearly have $(1-2\eta \sigma^2){\sigma} \leq \sqrt{2\twonorm{\M}}$. On the other hand, for $\sigma > \sqrt{2\twonorm{\M}}$, we have $(1-2\eta\sigma^2){\sigma} < (1-4\eta \twonorm{\M}) {\sigma}$. Plugging this observation into Eq.\eqref{eqn:Utspectralbound}, we obtain:
\begin{align*}
\twonorm{\Utp} &\leq \max\left(\sqrt{2\twonorm{\M}},(1-4\eta \twonorm{\M}) \twonorm{\Ut}\right) + 2 \eta \twonorm{\M} \twonorm{\Ut} \\ &\leq \max\left(\sqrt{2\twonorm{\M}}+\frac{1}{15}\twonorm{\Ut[t]},\twonorm{\Ut}\right)
\leq \max\left(\twonorm{\Ut[0]}, \sqrt{3 \twonorm{\M}}\right),
\end{align*}
where we used the inductive hypothesis in the last step. This proves the lemma.
\end{proof}

Finally, we prove the smoothness and gradient dominance in above regions.
\begin{lemma}[Restatement of Lemma \ref{lem:smoothness}] \label{lem:smoothness-re}
For any $\Ut[1], \Ut[2] \in \{\mat{U} | \twonorm{\mat{U}}^2 \le \Gamma\}$, we have function $f(\U) = \frob{\M - \U^2}^2$ satisfying:
\begin{equation}
\|\nabla f(\Ut[1]) - \nabla f(\Ut[2])\|_F \le 8 \max\{\Gamma, \twonorm{\M}\} \|\Ut[1] - \Ut[2]\|_F
\end{equation}
\end{lemma}

\begin{proof}
By expanding gradient $\nabla f(\U)$, and reordering terms, we have:
\begin{align*}
 & \|\nabla f(\Ut[1]) - \nabla f(\Ut[2])\|_F \\
 = &\|(2\Ut[1]^3 - \M\Ut[1] - \Ut[1]\M) - (2\Ut[2]^3 - \M\Ut[2] - \Ut[2]\M) \|_F \\
 = &\|2 (\Ut[1]^3 - \Ut[2]^3) - \M (\Ut[1] - \Ut[2]) -(\Ut[1] - \Ut[2])\M\|_F \\
 \le & 2\|\M\|_2\|\Ut[1] - \Ut[2]\|_F + 2\|\Ut[1]^3 - \Ut[2]^3\|_F \\
 = &2\|\M\|_2\|\Ut[1] - \Ut[2]\|_F + 2\|\Ut[1]^2(\Ut[1] - \Ut[2]) + \Ut[1](\Ut[1] - \Ut[2])\Ut[2] + (\Ut[1] - \Ut[2]) \Ut[2]^2\|_F \\
 \le & 2\|\M\|_2\|\Ut[1] - \Ut[2]\|_F + 6\Gamma \|\Ut[1] - \Ut[2]\|_F \\
 \le & 8 \max\{\Gamma, \twonorm{\M}\}\|\Ut[1] - \Ut[2]\|_F
\end{align*}
\end{proof}

\begin{lemma}[Restatement of Lemma \ref{lem:gradientdominate}] \label{lem:gradientdominate-re}
For any $\U \in \{\mat{U} | \singmin{\mat{U}}^2 \ge \gamma\}$, we have function $f(\U) = \frob{\M - \U^2}^2$ satisfying:
\begin{equation}
\|\nabla f(\U)\|^2_F \ge 4\gamma f(\U)
\end{equation}
\end{lemma}

\begin{proof}
By expanding gradient $\nabla f(\U)$, we have:
\begin{align*}
&\|\nabla f(\U)\|^2_F  = \|(\U^2 - \M)\U + \U(\U^2 - \M)\|^2_F \\
=& \langle (\U^2 - \M)\U + \U(\U^2 - \M), (\U^2 - \M)\U + \U(\U^2 - \M) \rangle \\
\ge & 4\sigma^2_{\min}(\U) \|\U^2 - \M\|_F^2 = 4\gamma f(\U)
\end{align*}

\end{proof}

%!TEX root = noncvx_gradDes.tex
\section{Proof of Theorem~\ref{thm:stability}}
In this section, we will prove Theorem~\ref{thm:stability}. We first state a useful lemma which is a stronger version of Lemma \ref{lem:smalleigval}
\begin{lemma}\label{lem:smalleigval-strong}
Suppose $\Ut$ is a PD matrix with $\twonorm{\Ut} \leq \max(\twonorm{\Ut[0]},\sqrt{3\twonorm{\M}})$, and $\singmin{\Ut} \geq \min(\singmin{\Ut[0]},\frac{1}{10}\sqrt{\singmin{\M}})$. Suppose further that $\eta < \frac{c \min(\singmin{\Ut[0]},\frac{1}{10}\sqrt{\singmin{\M}})}{ \max(\twonorm{\Ut[0]},\sqrt{3\twonorm{\M}})^{3/2} }$, where $c$ is a small enough constant and denote $\Utp \eqdef \Ut - \eta \left(\Ut^2 - \M\right) \Ut - \eta \Ut\left(\Ut^2 - \M\right)$. Then, $\Utp$ is a PD matrix with:
\begin{align*}
\lammin{\Utp} \ge \left(1+ \frac{\eta {\singmin{\M}}}{15}\right) \min(\singmin{\Ut[0]}, \frac{1}{10} \sqrt{\singmin{\M}}).
\end{align*}
\end{lemma}

Indeed, our proof of Lemma \ref{lem:smalleigval} already proves this stronger result. 
Now we are ready to prove Theorem \ref{thm:stability}.

\begin{proof}[Proof of Theorem~\ref{thm:stability}]
The proof of the theorem is a fairly straight forward modification of the proof of Theorem~\ref{thm:geomconv}. We will be terse since for most part we will use the arguments employed in the proofs of Theorem~\ref{thm:geomconv} and Lemmas~\ref{lem:opnorm-bound-MS} and~\ref{lem:smalleigval}.

We have the following two claims, which are robust versions of Lemmas~\ref{lem:opnorm-bound-MS} and~\ref{lem:smalleigval}, bounding the spectral norm and smallest eigenvalue of intermediate iterates. The proofs will be provided after the proof of the theorem.
\begin{claim}\label{claim:opnorm-bound-MS}
For every $t\in[T-1]$, we have:
\begin{align*}
\twonorm{\Ut} \leq \max(\twonorm{\Ut[0]},\sqrt{3 \twonorm{\M}}).
\end{align*}
\end{claim}

\begin{claim}\label{claim:smalleigval}
For every $t \in [T-1]$, we have $\Ut$ be a PD matrix with
\begin{align*}
\singmin{\Ut} \geq \min(\singmin{\Ut[0]},\frac{\sqrt{\singmin{\M}}}{10}).
\end{align*}
\end{claim}

We prove the theorem by induction. The base case $t=0$ holds trivially. Assuming the theorem is true for $t$, we will show it for $t+1$. Denoting $\Utptild \eqdef \Ut - \eta \left(\Ut^2 - \M\right) \Ut - \eta \Ut\left(\Ut^2 - \M\right)$, we have
\begin{align}
\frob{\M-\Utp^2} &= \frob{\M - \Utptild^2 - \Utptild \Errt - \Errt \Utptild - \Errt^2} \nn \\
&\leq \frob{\M - \Utptild^2} + 2 \twonorm{\Utptild} \frob{\Errt} + \frob{\Errt^2}. \label{eqn:errexp}
\end{align}
Using Claims~\ref{claim:opnorm-bound-MS} and~\ref{claim:smalleigval}, Lemma~\ref{lem:opnorm-bound-MS} tells us that
\begin{align*}
\twonorm{\Utptild} \leq \max(\twonorm{\Ut[0]},\sqrt{3 \twonorm{\M}}),
\end{align*}
and Theorem~\ref{thm:geomconv} tells us that
\begin{align*}
\frob{\M - \Utptild^2} \leq \exp\left(-\ch \eta \min(\singmin{\Ut[0]}^2,\singmin{\M})\right) \frob{\M - \Ut^2}.
\end{align*}
Plugging the above two conclusions into~\eqref{eqn:errexp}, tells us that
\begin{align*}
&\frob{\M-\Utp^2} \\
&\leq \exp\left(-\ch \eta \min(\singmin{\Ut[0]}^2,\singmin{\M})\right) \frob{\M - \Ut^2} + 2 \max(\twonorm{\Ut[0]},\sqrt{3 \twonorm{\M}}) \frob{\Errt} \\
&\qquad + \frac{1}{30} \eta \singmin{\M} \min(\singmin{\Ut[0]},\sqrt{\singmin{\M}}) \frob{\Errt} \\
&\leq \exp\left(-\ch \eta \min(\singmin{\Ut[0]}^2,\singmin{\M})\right) \frob{\M - \Ut[0]^2} \\
&\qquad + 4 \max(\twonorm{\Ut[0]},\sqrt{3 \twonorm{\M}}) \sum_{s=0}^{t} \exp\left(- \ch \eta \min(\singmin{\Ut[0]}^2,\singmin{\M}) (t-s) \right) \frob{\Errt[s]},
\end{align*}
where we used the induction hypothesis in the last step.
\end{proof}

We now prove Claim~\ref{claim:opnorm-bound-MS}.
\begin{proof}[Proof of Claim~\ref{claim:opnorm-bound-MS}]
Just as in the proof of Lemma~\ref{lem:opnorm-bound-MS}, we will use induction. Assuming the claim is true for $\Ut$, by update equation 
\begin{equation*}
\Utp = \Ut - \eta \left(\Ut^2 - \M\right) \Ut - \eta \Ut\left(\Ut^2 - \M\right) + \Errt
\end{equation*}
We can write out:
\begin{align}
\twonorm{\Utp}
%&= \twonorm{\Ut - \eta \left(\Ut^2 - \M\right) \Ut - \eta \Ut \left(\Ut^2 - \M\right) + \Errt} \nn\\
%&= \twonorm{\left(\eye - 2 \eta \Ut^2\right) \Ut + \eta \M \Ut + \eta \Ut \M + \Errt} \nn\\
%&\leq \twonorm{\left(\eye - 2 \eta \Ut^2\right) \Ut} + \eta \twonorm{\M \Ut} + \eta \twonorm{\Ut \M} + \twonorm{\Errt} \nn \\
&\leq \twonorm{\left(\eye - 2 \eta \Ut^2\right) \Ut} + \eta \twonorm{\M} \twonorm{\Ut} + \eta \twonorm{\Ut} \twonorm{\M} + \twonorm{\Errt}. \label{eqn:Utspectralbound-err}
\end{align}

Since $\eta < \frac{1}{10\max(\twonorm{\Ut[0]}^2,\twonorm{\M})}$ and $\twonorm{\Ut} \leq \max(\twonorm{\Ut[0]},\sqrt{3\twonorm{\M}})$, note that the singular values of the matrix $\left(\eye-2\eta \Ut^2\right)\Ut$ are exactly $(1-2\eta \sigma^2)\cdot{\sigma}$ where $\sigma$ is a singular value of $\Ut$. For $\sigma \leq \sqrt{2 \twonorm{\M}}$, we clearly have $(1-2\eta \sigma^2){\sigma} \leq \sqrt{2\twonorm{\M}}$. On the other hand, for $\sigma > \sqrt{2\twonorm{\M}}$, we have $(1-2\eta\sigma^2){\sigma} < (1-4\eta \twonorm{\M}) {\sigma}$. Plugging this observation into~\eqref{eqn:Utspectralbound-err}, we obtain:
\begin{align*}
\twonorm{\Utp} &\leq \max\left(\sqrt{2\twonorm{\M}},(1-4\eta \twonorm{\M}) \twonorm{\Ut}\right) + 2 \eta \twonorm{\M} \twonorm{\Ut} + \twonorm{\Errt} \\
&\leq \max(\twonorm{\Ut[0]},\sqrt{3 \twonorm{\M}}),
\end{align*}
proving the claim.
\end{proof}

We now prove Claim~\ref{claim:smalleigval}.
\begin{proof}[Proof of Claim~\ref{claim:smalleigval}]
We will use induction, with the proof following fairly easily using Lemma~\ref{lem:smalleigval-strong}. Suppose $\singmin{\Ut} \geq \min(\singmin{\Ut[0]}, \frac{1}{10} \sqrt{\singmin{\M}})$. Denoting
\begin{align*}
\Utptild \eqdef \Ut - \eta \left(\Ut^2 - \M\right) \Ut - \eta \Ut\left(\Ut^2 - \M\right),
\end{align*}
Lemma~\ref{lem:smalleigval-strong} tells us that
\begin{align*}
\singmin{\Utptild} \geq \left(1+ \frac{\eta {\singmin{\M}}}{15}\right) \min(\singmin{\Ut[0]}, \frac{1}{10} \sqrt{\singmin{\M}}),
\end{align*}
which then implies the claim, since
\begin{align*}
\singmin{\Utp} \geq \singmin{\Utp} - \twonorm{\Errt} \geq \min(\singmin{\Ut[0]}, \frac{1}{10} \sqrt{\singmin{\M}}).
\end{align*}
\end{proof}

%!TEX root = noncvx_gradDes.tex

\section{Proof of Theorem~\ref{thm:lb}}\label{sec:lb}
In this section, we will prove Theorem~\ref{thm:lb}.

\begin{proof}
Consider two-dimensional case, where 
\begin{equation*}
\M = \begin{pmatrix}
\twonorm{\M} & 0 \\
0 & \singmin{\M}
\end{pmatrix}
\end{equation*}

We will prove Theorem~\ref{thm:lb} by considering two cases of step size (where $\eta \ge \frac{1}{4\twonorm{\M}}$ or $\eta < \frac{1}{4\twonorm{\M}}$) separately. 
% The first statement tells us the stepsize $\eta$ can not be as large as $\Omega(\frac{1}{\twonorm{\M}})$, otherwise, the local search can stuck at saddle point. Then, the second statement shows $\Omega(\cnM)$ number of steps is unavoidable for local search algorithm to converge.

\paragraph{Case 1}: For step size $\eta \ge \frac{1}{4\twonorm{\M}}$.
Let $\beta = \frac{1}{2\eta\twonorm{\M}} + 1$, and
consider following initialization $\Ut[0]$:
\begin{equation*}
\Ut[0] = \begin{pmatrix}
\sqrt{\beta\twonorm{\M}} & 0 \\
0 & \sqrt{\singmin{\M}}
\end{pmatrix}
\end{equation*}
Since $\eta \ge \frac{1}{4\twonorm{\M}}$, we know $\beta \le 3$, which satisfies our assumption about $\Ut[0]$ and $\M$.
By calculation, we have:
\begin{equation*}
\Ut[0](\Ut[0]^2 - \M) + (\Ut[0]^2 - \M)\Ut[0] =\begin{pmatrix}
2(\beta - 1)\sqrt{\beta \twonorm{\M}^3} & 0 \\
0 & 0
\end{pmatrix}
\end{equation*}
and since $2\eta(\beta-1)\twonorm{\M} = 1$, we have:
\begin{equation*}
\Ut[1] = \Ut[0] - \eta[\Ut[0](\Ut[0]^2 - \M) + (\Ut[0]^2 - \M)\Ut[0]] = 
\begin{pmatrix}
\sqrt{\beta\twonorm{\M}} - 2\eta(\beta - 1)\sqrt{\beta \twonorm{\M}^3}  & 0 \\
0 & 1
\end{pmatrix} = 
\begin{pmatrix}
0  & 0 \\
0 & 1
\end{pmatrix} 
\end{equation*}
Then, by induction we can easily show for all $t\ge 1$, $\Ut = \Ut[1]$, thus $\frob{\Ut^2 - \M} \ge \twonorm{\M}\ge \frac{1}{4}\singmin{\M}$ .

\paragraph{Case 2}: For step size $\eta \le \frac{1}{4\twonorm{\M}}$, 
consider following initialization $\Ut[0]$:
\begin{equation*}
\Ut[0] = \begin{pmatrix}
\sqrt{\twonorm{\M}} & 0 \\
0 & \frac{1}{2}\sqrt{\singmin{\M}}
\end{pmatrix}
\end{equation*}
According to the update rule in Algorithm \ref{algo:gradDesMS}, we can easily show by induction that: for any $t \ge 0$, $\Ut$ is of form:
\begin{equation*}
\Ut = \begin{pmatrix}
\sqrt{\twonorm{\M}} & 0 \\
0 & \alpha_t\sqrt{\singmin{\M}}
\end{pmatrix}
\end{equation*}
where $\alpha_t$ is a factor that depends on $t$, satisfying $0\le \alpha_t \le 1$ and:
\begin{equation*}
\alpha_{t+1} = \alpha_t [1+ \eta \singmin{\M}(1-\alpha_t^2)], \quad \alpha_0 = \frac{1}{2}
\end{equation*}
Since $\eta \le \frac{1}{4\twonorm{\M}}$, we know:
\begin{equation*}
\alpha_{t+1} \le \alpha_t [1 + \frac{1}{4\cnM}(1-\alpha_t^2)]
\le \alpha_t + \frac{1}{4\cnM}
\end{equation*}
Therefore, for all 
$t \le \cnM$, we have $\alpha_t \le \frac{3}{4}$, and thus 
$\frob{\Ut[t]^2 - \M} \ge \frac{1}{4}\singmin{\M}$.
\end{proof}

\end{document}